\documentclass[11pt]{article}

\usepackage{latexsym}
\usepackage{amssymb}
\usepackage{amsthm}
\usepackage{amscd}
\usepackage{amsmath}
\usepackage{mathrsfs}
\usepackage{graphicx}
\usepackage{hyperref}
\usepackage{tabmac}
\usepackage{shuffle}

\usepackage[all]{xy}
\input xy \xyoption{frame}
\xyoption{dvips}

\usepackage[colorinlistoftodos]{todonotes}
\usepackage{fullpage}

\theoremstyle{definition}
\newtheorem* {theorem*}{Theorem}
\newtheorem* {conjecture*}{Conjecture}
\newtheorem{theorem}{Theorem}[section]

\newtheorem{propdef}[theorem]{Proposition-Definition}

\theoremstyle{definition}

\newtheorem* {example*}{Example}

\newtheorem{lemma}[theorem]{Lemma}
\theoremstyle{definition}
\newtheorem{definition}[theorem]{Definition}
\theoremstyle{definition}

\newtheorem{conjecture}[theorem]{Conjecture}
\newtheorem{proposition}[theorem]{Proposition}
\newtheorem{corollary}[theorem]{Corollary}

\newtheorem *{remark}{Remark}
\theoremstyle{definition}
\newtheorem {example}[theorem]{Example}
\theoremstyle{definition}

\theoremstyle{definition}

\theoremstyle{definition}

\xyoption{dvips}

\def\wh{\widehat}
\def\({\left(}
\def\){\right)}

\newcommand{\CC}{\mathbb{C}}

\newcommand{\cP}{\mathcal{P}}

\newcommand{\cR}{\mathcal{R}}

\newcommand{\cT}{\mathcal{T}}
\newcommand{\cI}{\mathcal{I}}

\newcommand{\cD}{\mathcal{D}}
\newcommand{\cZ}{\mathcal{Z}}

\def\NN{\mathbb{N}}

\def\CC{\mathbb{C}}

\def\ZZ{\mathbb{Z}}

\def\GL{\mathrm{GL}}

\newcommand{\supp}{\mathrm{supp}}

\newcommand{\cN}{\mathcal{N}}
\newcommand{\cL}{\mathcal{L}}

\def\fk{\mathfrak}

\def\barr{\begin{array}}
\def\earr{\end{array}}
\def\ba{\begin{aligned}}
\def\ea{\end{aligned}}
\def\be{\begin{equation}}
\def\ee{\end{equation}}

\def\Cyc{\mathrm{Cyc}}

\def\qquand{\qquad\text{and}\qquad}

\def\qquord{\qquad\text{or}\qquad}

\def\inv{\mathrm{Inv}}

\def\I{\mathcal{I}}

\def\DesR{\mathrm{Des}_R}
\def\DesL{\mathrm{Des}_L}

\def\ds{\displaystyle}

\def\PP{\mathbb{P}}

\def\fkS{\fk S}

\def\ben{\begin{enumerate}}
\def\een{\end{enumerate}}

\def\cT{\mathcal{T}}
\def\cE{\mathcal E}

\def\fpf{{\tt {FPF}}}

\newcommand{\wfpf}{\Theta}
\def\DesF{\mathrm{Des}_\fpf}

\def\ellhat{\hat\ell}

\def\Sfpf{\hat {\fk S}^\fpf}
\def\ifpf{\iota}

\def\a{\textbf{a}}
\def\b{\textbf{b}}
\def\c{\textbf{c}}

\newcommand{\xRightarrow}[2][]{\ext@arrow 0359\Rightarrowfill@{#1}{#2}}

\newcommand{\Fl}{\operatorname{Fl}}

\renewcommand{\O}{\operatorname{O}}
\newcommand{\Sp}{\operatorname{Sp}}

\newcommand{\cA}{\mathcal{A}}
\newcommand{\cB}{\mathcal{B}}
\newcommand{\iB}{\hat\cB}
\newcommand{\fB}{\hat\cB^\fpf}

\def\cAfpf{\cA_\fpf}
\def\cRfpf{\hat\cR_\fpf}
\def\iR{\hat\cR}

\def\F{\mathcal{F}}
\newcommand{\arc}[2]{ \ar @/^#1pc/ @{-} [#2] }
\def\arcstop{\endxy\ }
\def\arcstart{\ \xy<0cm,-.15cm>\xymatrix@R=.1cm@C=.3cm }
\newcommand{\arcstartc}[1]{\ \xy<0cm,-.15cm>\xymatrix@R=.1cm@C=#1cm}

\def\ellhat{\hat\ell}

\def\Sfpf{\hat {\fk S}^\fpf}
\def\iS{\hat \fkS}

\newcommand{\ct}{\operatorname{ct}}

\newcommand{\del}{\operatorname{del}}

\newcommand{\push}{\mathord{\downarrow}}
\newcommand{\pull}{\mathord{\uparrow}}

\newcommand{\ipush}{\mathord{\Downarrow}}
\newcommand{\ipull}{\mathord{\Uparrow}}

\def\ellfpf{\hat\ell_\fpf}
\def\Pfpf{\hat\Psi}

\let\bxd\boxed
\renewcommand{\boxed}[1]{\hspace{0.7pt}\bxd{\!#1\!}\hspace{0.7pt}}

\numberwithin{equation}{section}
\allowdisplaybreaks[1]
\UseCrayolaColors

\makeatletter
\renewcommand{\@makefnmark}{\mbox{\textsuperscript{}}}
\makeatother

\begin{document}
\title{Transition formulas for involution Schubert polynomials}

\author{
    Zachary Hamaker\footnote{This author was supported by the IMA with funds provided by the National Science Foundation.} \\
    { \tt zachary.hamaker@gmail.com}
 \and
    Eric Marberg\footnote{This author was supported through a fellowship
    from the National Science Foundation.} \\
    {\tt eric.marberg@gmail.com}
\vspace{3mm}
\and
    Brendan Pawlowski\footnote{This author was partially supported by NSF grant 1148634.} \\
    {\tt br.pawlowski@gmail.com}
}

\date{}

\maketitle

\begin{abstract}
The orbits of the orthogonal and symplectic groups on the flag variety
are in bijection, respectively, with the involutions and fixed-point-free involutions 
in the symmetric group $S_n$. Wyser and Yong have described polynomial representatives
for the cohomology classes of the closures of these orbits,
which we denote as $\hat \fkS_y$ (to be called \emph{involution Schubert polynomials})
and $\Sfpf_y$ (to be called \emph{fixed-point-free involution Schubert polynomials}).
Our main results  are explicit formulas decomposing the product of $\hat \fkS_y$ (respectively, $\Sfpf_y$)
with any $y$-invariant linear polynomial as a linear combination of other
involution Schubert polynomials. These identities serve as analogues of
Lascoux and Sch\"utzenberger's transition formula for Schubert polynomials, and
lead to a self-contained algebraic proof of the nontrivial equivalence of several definitions of 
 $\hat \fkS_y$ and $ \Sfpf_y$ appearing in the literature.
Our formulas also imply combinatorial identities about \emph{involution words},
certain variations of reduced words  for involutions in $S_n$. 
We construct operators on involution words based on the Little map to
prove these identities bijectively.
The proofs of our main theorems depend on some new technical results, extending work of Incitti,
about  covering relations in the Bruhat order of $S_n$ restricted to involutions.
\end{abstract}

\tableofcontents
\setcounter{tocdepth}{2}

\section{Introduction}

Let $S_\ZZ$ denote the group of permutations of $\ZZ$ which fix all but finitely many points,
and write $S_\infty$ for the subgroup of elements in $S_\ZZ$ with support contained in 
$\PP = \{1,2,3,\dots\}$.
Define  $\I_\infty$ (respectively, $\I_\ZZ$) as the subset of involutions in $S_\infty$
(respectively, $S_\ZZ$).
We also write $S_n$ and $\I_n$ for the subsets of $S_\infty$
and $\I_\infty$ which fix all numbers outside $[n]=\{1,2,\dots,n\}$, and $\F_n \subset \I_n$ for
the subset of fixed-point-free involutions. The \emph{Schubert polynomials} are a family
of homogeneous polynomials $\fkS_w \in \ZZ[x_1, x_2, \ldots]$ indexed by $w \in S_\infty$.
Write $B $ for the subgroup of lower triangular matrices in $\GL_n(\CC)$. It is
well-known that the right $B$-orbits in the flag variety $\Fl(n) = B \backslash \GL_n(\CC)$ are in bijection
with $S_n$, that the integral cohomology ring of $\Fl(n)$ is isomorphic
to a quotient of $\ZZ[x_1,x_2, \ldots, x_n]$, and that under this isomorphism, the Schubert
polynomials $\{\fkS_w : w \in S_n\}$ correspond to the cohomology classes Poincar\'e dual
to the closures of the aforementioned $B$-orbits; see \cite{Manivel} for details.

The \emph{involution Schubert polynomials} are homogeneous polynomials $\iS_y$
indexed by $y \in \I_\infty$ serving a similar geometric purpose: the right
orbits of $\O_n(\CC)$ on $\Fl(n)$ are in bijection with $\I_n$,
and the cohomology classes of their orbit closures are (up to a constant factor)
represented by the involution Schubert polynomials $\{\iS_y : y \in \I_n\}$.
The family of \emph{fixed-point-free involution Schubert polynomials} $\{\Sfpf_z : z \in \F_{n}\}$
plays an analogous role when $n$ is even and $\O_n(\CC)$ is replaced by $\Sp_{n}(\CC)$.
The precise definitions of $\iS_y$ and $\Sfpf_z$ appear in Sections~\ref{invtrans-sect} and \ref{fpfinvtrans-sect}.
We attribute the definitions of these polynomials to Wyser and Yong \cite{WY},
although
they occur
as special cases of the cohomology representatives described in older work of Brion \cite[Theorem 1.5]{Brion98}.
It is Wyser and Yong's more explicit construction in terms of divided difference operators, however,
that is the real starting point of our results.

Besides their geometric significance, Schubert polynomials are important in combinatorics,
and our goal here is to find analogues of some classical Schubert combinatorics in the involution setting.
Let $<$ denote the (strong) Bruhat order on $S_\ZZ$.
The \emph{transition formula} of Lascoux and Sch\"utzenberger \cite{LS} expresses a product $x_r \fkS_w$
as a linear combination of Schubert polynomials. To be more specific, given $y \in S_{\ZZ}$ and $r \in \ZZ$, define
\begin{align*}
&\Phi^+(y,r) = \{w \in S_\ZZ : \text{$y \lessdot w$ and $w = y(r,j)$ where $r < j$}\}\\
&\Phi^-(y,r) = \{w \in S_\ZZ : \text{$y \lessdot w$ and $w = y(i,r)$ where $i < r$}\},
\end{align*}
where $y \lessdot w$ indicates that $w$ covers $y$ in Bruhat order, i.e., $\{y \} = \{ \sigma \in S_\ZZ : y \leq \sigma < w\}$.
Let $\fkS_w$ for $w \in S_\infty$ be defined as in Section~\ref{divided-sect}.

\begin{theorem}[See \cite{LS}] \label{thm:ordinary-transition-formula} If $y \in S_{\infty}$ and $r \in \PP$
then
$
x_r \fkS_y = \sum_{w \in \Phi^+(y,r)} \fkS_w - \sum_{w \in \Phi^-(y,r)} \fkS_w
$
 where we set $\fkS_w = 0$ for $w \in S_{\ZZ} - S_{\infty}$.
\end{theorem}

\begin{example}
If  (in one-line notation) $y=4153726 \in S_7$, then $\Phi^+(y,2) = \{y(2,3), y(2,4), y(2,6)\}$
and $\Phi^-(y,2) = \{ y(0,2)\}$,
so  $x_4 \fkS_{4153726} = \fkS_{4513726} + \fkS_{4351726} + \fkS_{4253716}$.
\end{example}

One of our main results is an analogous \emph{involution transition formula}.
To state this,  we require a brief digression about the Bruhat order on involutions.
Write $\lessdot_\I$ for the 
covering relation in the Bruhat order on $S_\ZZ$ restricted to $\I_\ZZ$,
so that  $y\lessdot_\I z$ if and only if $y,z \in \I_\ZZ$ and $\{y\} = \{ \sigma \in \I_\ZZ : y \leq \sigma < z\}$.
The permutations covering a given element $y \in S_\ZZ$ in the usual Bruhat order
are naturally labeled by transpositions, since by definition  if $y \lessdot w$ then  $w = y(i,j)$ for unique integers $i<j$.
In Section~\ref{bruhat-sect} we will describe
an equally natural though much less obvious  method of labeling the covering relations in  $(\I_\ZZ,<)$ by transpositions.
We sketch the main ideas here, in order to 
define the appropriate substitutes for the sets $\Phi^\pm(y,r)$ in our transition formula for $\iS_y$.

Write $\ell : S_\ZZ \to \NN$ for the usual length function on the symmetric group.
The \emph{Demazure product} on $S_\ZZ$ is the unique associative map
$\circ : S_\ZZ \times S_\ZZ \to S_\ZZ$ such that
$u\circ v = uv$ if $\ell(uv) = \ell(u) +\ell(v)$ and $s\circ s = s$ for all simple transpositions $s \in S_\ZZ$.
One can show that $\I_\ZZ = \{ w^{-1} \circ w : w \in S_\ZZ\}$,
and we define $\cA(y)$ for $y \in \I_\ZZ$ as the set of permutations $w \in S_\ZZ$ of minimal length such that 
$y = w^{-1}\circ w$.
More background on these sets and their properties is presented in Section~\ref{invword-sect}.
The following statement  is equivalent to Theorem~\ref{tau-thm} and gives one of our key technical results.

\begin{theorem}\label{tauintro-thm}
Let $i<j$ be distinct integers and set $t = (i,j) \in S_\ZZ$.
For each $y \in \I_\ZZ$, there exists at most one involution $z \in \I_\ZZ$
such that $\varnothing \neq \{ wt : wt \lessdot w \in \cA(z) \} \subset \cA(y)$.
\end{theorem}

In Section \ref{bruhat-sect} we explicitly construct, for any integers $i<j$,  a map $\tau_{ij} : \I_\ZZ \to \I_\ZZ$
with the property that
if $ y,z \in \I_\ZZ$ and $v \in \cA(y)$ are such that $v\lessdot v(i,j)$ and $v(i,j) \in \cA(z)$,
then $z = \tau_{ij}(y)$.  The given property  does not uniquely determine $\tau_{ij}$,
but the \emph{a priori} nontrivial claim that such a map exists is equivalent to Theorem \ref{tauintro-thm}.
The maps $\tau_{ij}$ are slightly more general versions of the
\emph{covering transformations} which Incitti defines in \cite{Incitti1}.
Crucially, as first noted in Incitti's work, these transformations completely describe the Bruhat covers in $\I_\ZZ$ in the following sense:
 
 \begin{theorem}[Incitti \cite{Incitti1}]
If $y,z \in \I_\ZZ$ are such that $y\lessdot_\I z$ then $z = \tau_{ij}(y)$ for some $i<j$.
\end{theorem}

See Theorem~\ref{taubruhat-thm} for a stronger formulation of this result.
We may at last describe our involution transition formula.
For $y \in \I_\ZZ$ and $r \in \ZZ$,  define
\be\label{hatphi-eq}
\ba
\hat\Phi^+(y,r) &= \{z \in \I_\ZZ : \text{$y \lessdot_\I z$ and $z = \tau_{rj}(y)$ for an integer $j>r$}\}\\
\hat\Phi^-(y,r) &= \{z \in \I_\ZZ : \text{$y \lessdot_\I z$ and $z = \tau_{ir}(y)$ for an integer $i < r$}\}.
\ea
\ee
Let $\Cyc_\PP(y) = \{ (p,q) \in \PP\times \PP :  p\leq q = y(p)\}$,
and for   $p,q  \in\PP$ define $x_{(p,q)}$ to be  $x_p+ x_q$ if $p\neq q$ and $x_p$ if $p=q$.
Let $\iS_y$ for $y \in \I_\infty$ be given as in
Definition \ref{iS-def}. We prove the following identity in Section~\ref{itransition-sect}.

\begin{theorem} \label{thm:involution-transition-formula}
If $y \in \I_\infty$ and $(p,q) \in \Cyc_\PP(y)$ then 
$x_{(p,q)}\iS_y = \sum_{z \in \hat\Phi^+(y,q)} \iS_z - \sum_{z \in \hat\Phi^-(y,p)} \iS_z$
where we set $\iS_z = 0$ for $z \in \I_\ZZ-\I_\infty$.
\end{theorem}

\begin{example}
If $y=(2,3)(4,7) \in \I_7$ then one can show (see Definition~\ref{tau-def}) that
\[
\ba
\hat\Phi^+(y,3)
&= \{ \tau_{3,4}(y),\ \tau_{3,5}(y),\ \tau_{3,7}(y)\}
= \{ (2,4)(3,7),\ (2,5)(4,7),\ (2,7)\}
\\
\hat\Phi^-(y,2) &= \{\tau_{1,2}(y)\} = \{ (1,3)(4,7)\}
\ea
\]
and so $(x_2+x_3)\iS_{(2,3)(4,7)} = \iS_{(2,4)(3,7)}+\iS_{(2,5)(4,7)}+\iS_{(2,7)}-\iS_{(1,3)(4,7)}$.
\end{example}

Theorem~\ref{thm:fpf-transition-formula} presents a similar transition formula for the fixed-point-free
involution Schubert polynomials $\Sfpf_y$ specified by Definition~\ref{fS-def}.
(For brevity, we omit 
 the precise statement in this introduction.)
These results only tell us how to decompose
products of $\iS_y$ (respectively, $\Sfpf_y$)
with $y$-invariant linear polynomials. One cannot hope to do much better, however:
unlike ordinary Schubert polynomials, involution Schubert polynomials do not span
$\ZZ[x_1, x_2, \ldots]$, and one can check that, for instance, $x_1 \iS_{(1,2)}$ is not
a linear combination of involution Schubert polynomials.

\begin{remark}
Wyser and Yong \cite{WY2} have also described polynomial representatives for cohomology classes of 
the closures of 
the $\GL_p(\CC) \times \GL_{q}(\CC)$-orbits in the flag variety $\Fl(n)$, when $n=p+q$ .
It is an interesting open problem to find an analogous transition formula for these polynomials.
\end{remark}

Throughout, we write $s_i=(i,i+1) \in S_\ZZ$ for $i \in \ZZ$ to denote the  simple transposition
exchanging $i$ and $i+1$. A \emph{reduced word} for  $w \in S_\ZZ$ is a sequence of simple
transpositions $(s_{i_1}, s_{i_2},\dots,s_{i_k})$ of minimal possible length $k=\ell(w)$ such that
$w = s_{i_1} s_{i_2}\cdots s_{i_k}$.
Let $\cR(w)$ be the set of reduced words for $w \in S_\ZZ$.
One can show that Theorem~\ref{thm:ordinary-transition-formula} implies that the sets
$\bigcup_{w \in \Phi^+(y,r)} \cR(w)$ and $\bigcup_{w \in \Phi^-(y,r)} \cR(w)$ have the same cardinality,
and the \emph{Little map} described in \cite{Little} provides an explicit bijection.
The involution transition formula leads to similar results for the appropriate analogue
of reduced words. Namely,
an \emph{involution word} for $y \in \I_\infty$ is 
a sequence of simple
transpositions $(s_{i_1}, s_{i_2},\dots,s_{i_k})$  of minimal possible length such that
$y = s_{i_k} \circ \cdots \circ s_{i_2} \circ s_{i_1} \circ s_{i_2}\circ \cdots \circ s_{i_k}$; 
see Section~\ref{invword-sect} for more background on these objects.
Let $\iR(y)$ denote the set of involution words of $y$.
Theorem~\ref{thm:involution-transition-formula}
implies that the sets $\bigcup_{z \in \hat\Phi^+(y,q)} \iR(z)$ and $\bigcup_{z \in \hat\Phi^-(y,p)} \iR(z)$
have the same cardinality (see Proposition~\ref{little3-cor}), and in Section~\ref{little-sect}
we show that a modification of Little's algorithm provides an explicit bijection.
Section~\ref{little2-sect} presents a bijective proof of an analogous identity in the fixed-point-free case.

Finally, we mention some applications of Theorem~\ref{thm:involution-transition-formula} which will appear 
in the companion papers \cite{HMP4,HMP5}. For $w \in S_\ZZ$ and $N \in \ZZ$, write $w \gg N$ for the permutation
of $\ZZ$
defined by $i \mapsto w(i-N) + N$. It can be shown that the limit
$F_w = \lim_{N\to \infty} \fkS_{ w \gg N}$ exists as a formal power series, and is in fact a symmetric function---the
so-called \emph{Stanley symmetric function} of $w$. By taking limits in Theorem~\ref{thm:ordinary-transition-formula},
one obtains a recurrence
which can be used to prove the Schur-positivity of $F_w$ and effectively compute its Schur expansion,
which includes as a special case the Littlewood-Richardson rule. Similarly, for each $y \in \I_\infty$ there is an 
\emph{involution Stanley symmetric function} $\hat F_y = \lim_{N \to \infty} \iS_{y \gg N} = \sum_{w \in \cA(y)} F_w$.
Taking limits in Theorem~\ref{thm:involution-transition-formula} gives a recurrence which we use in \cite{HMP4}
to prove that $\hat F_y$ is a nonnegative linear combination of Schur $P$-functions, and which provides a new
Littlewood-Richardson rule for expanding the product of Schur $P$-functions in the Schur $P$-basis.

The structure of this paper is as follows. Section~\ref{sec:preliminaries} gives some preliminaries
on Schubert polynomials and involution words. In Section~\ref{invtrans-sect}, we prove Theorems~\ref{tauintro-thm}
and \ref{thm:involution-transition-formula}, as well as some related results on the Bruhat order, the polynomials $\iS_y$,
and the involution Little map.
Section~\ref{fpfinvtrans-sect} contains analogues of the results of Section~\ref{invtrans-sect} for
fixed-point-free involutions.
While the main results in Sections~\ref{invtrans-sect} and \ref{fpfinvtrans-sect} are formally similar,
our proofs in the two cases proceed by  distinct strategies.

\subsection*{Acknowledgements}

We thank
Dan Bump, Michael Joyce,
Vic Reiner,
Alex Woo,
Ben Wyser, and
Alex Yong for many helpful conversations during the development of this paper.
We also thank the anonymous referees for their useful comments and suggestions.

\section{Preliminaries} \label{sec:preliminaries}

For $w \in S_\ZZ$, let
 $\inv(w) = \{ (i,j) \in \ZZ\times \ZZ : i<j\text{ and }w(i)>w(j)\}$ 
so that  $\ell(w) = |\inv(w)|$.
We write $\DesL(w)$ and $\DesR(w)$ for the
\emph{left}  and \emph{right descent sets} of $w \in S_\ZZ$,
consisting of the simple transpositions $s_i = (i,i+1)$ such that $\ell(s_iw)<\ell(w)$ and
$\ell(ws_i)<\ell(w)$, respectively.  It is useful to recall that $s_i \in \DesR(w)$
for $w \in S_\ZZ$ if and only if $w(i)> w(i+1)$.

\subsection{Schubert polynomials}\label{divided-sect}

We recall some facts about divided difference operators and Schubert polynomials.
Let $\cP = \ZZ[x_1,x_2,\dots]$ be the ring of 
polynomials over $\ZZ$ in a countable set of commuting indeterminates.
The group $S_\infty$ acts on $\cP$ by permuting variables, and one sets
\[ \partial_i f = (f - s_i f)/(x_i-x_{i+1})\qquad\text{for $i \in \PP$ and $f \in \cP$}.\]
The \emph{divided difference operator} $\partial_i$ defines a $\ZZ$-linear map $\cP\to \cP$.
By definition, $\partial_i f = 0$ if and only if $s_i f = f$,
and if $f \in \cP$ is homogeneous  then $\partial_i f $ is either zero or
homogeneous of degree $\deg(f)-1$.
We note the following identity which  implies, in particular,
that $\partial_i(fg) = f \partial_i g$ if $\partial_i f = 0$:

\begin{lemma}\label{lem6}
If $i \in \PP$ and $f,g \in \cL$  then $\partial_i(fg) =(\partial_if)g +  (s_if) \partial_i g$.
\end{lemma}

The divided difference operators satisfy $\partial_i^2=0$ as well as the usual braid relations for
$S_\infty$, and so if $w \in S_\infty$ then $\partial_{i_1}\partial_{i_2}\cdots \partial_{i_k}$ is
the same map  $\cP\to \cP$  for all reduced words $(s_{i_1},s_{i_2},\dots,s_{i_k}) \in \cR(w)$.
We denote this map by
$\partial_w :\cP \to \cP$ for $w \in S_\infty.$
This notation affords the most succinct algebraic definition of 
the  \emph{Schubert polynomial}  $\fkS_v$ of a permutation $v \in S_n$,
namely:
\[\fkS_v = \partial_{v^{-1}w_n}x^{\delta_n} \in \cP\] where 
$w_n=n\cdots321 \in S_n$ is the reverse permutation and $x^{\delta_n} = x_1^{n-1} x_2^{n-2}\cdots  x_{n-1}^1.$
Contrary to appearances, this formula for $\fkS_v$ is independent of the choice of $n$
such that $v \in S_n$, as one can deduce by checking that
$\partial_{w_mw_{n}} x^{\delta_n} = x^{\delta_m}$
for positive integers $m<n$.
We may therefore consider the Schubert polynomials to be a family indexed by $S_\infty$.

Some useful
references on Schubert polynomials include \cite{BB,BJS,Knutson,Macdonald,Manivel}.
Since $\partial_i^2=0$, it follows directly from the definition that
\be\label{maineq}
\fkS_1=1
\qquand
\partial_i \fkS_w = \begin{cases} \fkS_{ws_i} &\text{if $s_i \in \DesR(w)$}
\\
0&
\text{if $s_i \notin \DesR(w)$}\end{cases}
\qquad\text{for each $i \in \PP$}.
\ee
Conversely,
one can show that
$\{ \fkS_w\}_{w \in S_\infty}$ is the unique family of homogeneous polynomials indexed by
$S_\infty$ satisfying \eqref{maineq}; see \cite[Theorem 2.3]{Knutson} or \cite{BH}.
One checks as an exercise that $\deg \fkS_w=\ell(w)$
and
$\fkS_{s_i} = x_1+x_2+\dots +x_i$ for $i \in \PP$.
We recall this less obvious fact \cite[Proposition 2.5.4]{Manivel}:

\begin{proposition}[See \cite{Manivel}]\label{basis-prop}
The polynomials $\fkS_w$ for $w \in S_\infty$ with $\DesR(w)\subset\{s_1,s_2,\dots,s_n\}$
form a $\ZZ$-basis for  $\ZZ[x_1,x_2,\dots,x_n]$.
\end{proposition}

As in the introduction, let $<$ denote the \emph{Bruhat order} on $S_\ZZ$, which by definition is the weakest 
strict partial order on $S_\ZZ$ with $w < wt$ whenever $t $
is a transposition and $\ell(w)<\ell(wt)$. 
Recall that we write $u \lessdot v$ for $u,v \in S_\ZZ$ if $\{u\} = \{ w \in S_\ZZ : u\leq w<v\}$.
The poset $(S_\ZZ, <)$ contains $S_\infty$ as a lower ideal and is graded with rank function
$\ell$. Consequently $u \lessdot v$  if and only if $u<v$ and $\ell(v) = \ell(u)+1$.
In applying this observation, the following well-known fact is useful.

\begin{lemma}\label{monk-lem}
If $u\in S_\ZZ$ and $t=(a,b) \in S_\ZZ$ for some integers $a<b$,
then $\ell(ut) =\ell(u)+1$ if and only if $u(a) < u(b)$ and no $i \in \ZZ$
exists with $a<i<b$ and $u(a)<u(i)<u(b)$.
\end{lemma}

\subsection{Involution words}\label{invword-sect}

Let $(W,S)$ be a Coxeter system with length function $\ell$, and define
$\I(W) = \{ y \in W : y=y^{-1}\}$.
When $W$ is $S_n$, $S_\infty$, or $S_\ZZ$,
we take $S$ to be the simple generating set $\{s_i : i \in \ZZ\} \cap W$.
Most of the material in this section appears in some form in work of
Richardson and Springer \cite{RichSpring, RichSpring2} or Hultman \cite{H1,H2,H3,H4};
our notation follows \cite{HMP2}. The following is well-known (cf.\ \cite{KM}).

\begin{propdef}\label{demazure-lem}
There exists a unique associative map (called the \emph{Demazure product})
$\circ : W \times W \to W$ such that
$u\circ v = uv$ if $\ell(uv) = \ell(u) +\ell(v)$ and $s\circ s = s$ for all $s \in S$.
\end{propdef}

Clearly $s\circ w = w\circ t  = w$ if $s \in \DesL(w)$ and $t \in \DesR(w)$,
where $\DesL(w)$ and $\DesR(w)$ denote the usual descent sets
of $w \in W$.
On the other hand, if $t_1,t_2,\dots,t_k \in S$ are such that
$(t_{1},t_{2},\dots,t_{k})$ is a reduced word for $w$ then
$w = t_{1} \circ t_{2} \circ \dots \circ t_{k} =
t_{1}  t_{2}  \dots  t_{k}  $.
As a consequence of these observations and the exchange principle for Coxeter systems,
one obtains the following lemma.

\begin{lemma}\label{demazure-lem2}
If $y \in \I(W)$ and $s \in S$ then
$ s\circ y \circ s =
\begin{cases}
sys &  \text{if } s \notin\DesR(y) \text{ and }sy \neq ys \\
ys  &  \text{if } s \notin\DesR(y) \text{ and }sy=ys \\
y   &  \text{if } s \in \DesR(y).
\end{cases}
$
\end{lemma}

Thus, if $y \in \I(W)$ then $s\circ y \circ s \in \I(W)$,
and by induction on length one may deduce:

\begin{corollary}
If $y \in \I(W)$ then $y = w^{-1} \circ w$  for some $w \in W$.
\end{corollary}

Given $y \in \I(W)$, let $\cA(y)$ denote the set of elements $w \in W$ of minimal length
such that $y = w^{-1}  \circ w$.
Define $\iR(y) = \bigcup_{w \in \cA(y)} \cR(w)$, so that $\iR(y)$  consists of all sequences
$(t_{1},t_{2},\dots ,t_{k})$ with $t_i \in S$ of minimal length $k$ such that
$y = t_{k}\circ \dots \circ t_{2} \circ t_{1} \circ t_{2} \circ \dots \circ t_{k}.$
We refer to the elements of $\cA(y)$ as \emph{atoms} of $y \in \I(W)$
and to the elements of $\iR(y)$ as \emph{involution words}.
These sets have been
 studied previously in \cite{CJW, HMP1, HMP2, HuZhang}.

\begin{example}\label{fkS-ex}
For $W=S_3$   we have $\iR(321) = \{ (s_1,s_2), (s_2,s_1)\}$ and $\cA(321) = \{ 231, 312\}$.
\end{example}

The following technical lemma is \cite [Proposition 2.8]{HMP2}.

\begin{lemma}[See \cite{HMP2}]
 \label{atom-lem} Let $y,z \in \I(W)$. The following properties then hold:
\ben
\item[(a)] If $s\in \DesR(z)$ and $z = s\circ y \circ s$ then
$\cA(y) = \{ ws : w \in \cA(z)\text{ such that }\ell(ws)=\ell(w)-1\}$.

\item[(b)] $\DesR(w)\subset\DesR(y)$ for all $w \in \cA(y)$.
\een
\end{lemma}

Let $\ellhat(y)$ denote the common length of the elements of $\iR(y)$ and
$\cA(y)$ for $y \in \I(W)$.
Evidently $\ellhat(s\circ y \circ s) = \ellhat(y)+1$ for $s \in S$ with $s \notin\DesR(y)$.
Let $\kappa(y) = 2\ellhat(y)-\ell(y)$ so that $\ellhat = \frac{1}{2}\( \ell+\kappa\)$.
The following is a straightforward exercise; see \cite{Incitti1}.

\begin{proposition}\label{ellhat-prop}
If $W \in \{ S_n, S_\infty, S_\ZZ\}$,
then $\kappa(y)$ is the number of 2-cycles of $y \in \I(W)$.
\end{proposition}

Recall that $\I_n$, $\I_\infty$, and $\I_\ZZ$ denote the sets of involutions (i.e., elements $y=y^{-1}$)
in $S_n$,  $S_\infty$, and $S_\ZZ$.
The involutions in these groups are the permutations whose cycles all have at most two elements.
It is often convenient to identify these permutations with the partial
matchings on  $[n]$, $\PP$, or $\ZZ$ in which
distinct vertices $i$ and $j$ are connected by an edge whenever they form a nontrivial cycle.
By convention, we draw such matchings so that the vertices are points on the horizontal real axis
and the edges appear as  convex curves in the upper half plane. For example,
\[
 (1,6)(2,7)(3,4) \in \I_7
 \qquad
 \text{is represented as the matching}
 \qquad
\arcstart{
*{.}  \arc{1.}{rrrrr}   & *{.}  \arc{1.}{rrrrr}   & *{.} \arc{.5}{r} & *{.} & *{.} & *{.} & *{.}
}\endxy
\]
We omit the numbers labeling the vertices in  matchings corresponding to
involutions in $\I_\infty$.

\section{Formulas in the orthogonal case}\label{invtrans-sect}

In this section we prove Theorem \ref{thm:involution-transition-formula} from the introduction,
which gives a transition formula for 
 the \emph{involution Schubert polynomials} $\{ \iS_y\}_{y \in \I_\infty}$.
 The simplest definition of these polynomials, leveraging the notation 
 in Section \ref{invword-sect}, goes as follows:
 
\begin{definition} \label{iS-def}
The \emph{involution Schubert polynomial} of $y \in \I_\infty$ is 
$ \iS_y = \sum_{w \in \cA(y)} \fkS_w.$
\end{definition}

\begin{example}
Using Example \ref{fkS-ex}, we have
$\iS_{321} = \fkS_{132} + \fkS_{312} = x_1^2  + x_1 x_2$.
\end{example}

The essential algebraic properties of these polynomials are given by \cite[Theorem 3.11]{HMP1}:

\begin{theorem}[See \cite{HMP1}]\label{ithm}
The involution Schubert polynomials $\{ \iS_y\}_{y \in \I_\infty}$
are the unique family of homogeneous polynomials indexed by $\I_\infty$
such that if $i \in \PP$ and $s=s_i$ then
\be\label{i-eq}
\iS_{1} = 1
\qquand
\partial_i \iS_y =
\begin{cases}
\iS_{sys} &\text{if $s \in \DesR(y)$ and $sy\neq ys$} \\
\iS_{ys}  &\text{if $s \in \DesR(y)$ and $sy=ys$} \\
0         &\text{if $s \notin \DesR(y)$.}
\end{cases}
\ee
\end{theorem}

\begin{remark}
Note that if $s_i \notin \DesR(x)$ then $\partial_i\iS_{s_i\circ y\circ s_i} = \iS_y$.
Since $\fkS_w$  has degree $\ell(w)$, it follows that $\iS_y$  has degree $\ellhat(y)$.
As the sets $\cA(y)$ for $y \in \I_\infty$ are pairwise disjoint,
the polynomials $\iS_y$ for $y \in \I_\infty$ are linearly independent by
Proposition \ref{basis-prop}.
It is an open problem, which we address only glancingly in this work, to describe the
$\ZZ$-module spanned by these polynomials.
\end{remark}

The involution Schubert polynomials were defined in a  rescaled form by Wyser and Yong
in \cite{WY}, where they were denoted $\Upsilon_{y; (\GL_n, \O_n)}$.  The precise relationship is
$2^{\kappa(y)} \iS_y = \Upsilon_{y; (\GL_n, \O_n)}$, although this identity is not obvious from the
definitions here and in \cite{WY}. (One way to confirm the identity is to check that the rescaled
$\Upsilon$-polynomials are a second family satisfying \eqref{i-eq}; see \cite[Section 3.4]{HMP1}.
Another approach will be discussed at the end of Section \ref{itransition-sect}.)
Wyser and Yong's definition was motivated by the study of the action of the orthogonal group
$\O_n(\CC)$ on the flag variety $\Fl(n) = B\backslash \GL_n(\CC)$. 
As noted in the introduction, the involution Schubert polynomials may be
identified with cohomology representatives of the closures of the $\O_n(\CC)$-orbits in the flag variety 
$\Fl(n)$, and are instances of a more general construction of Brion \cite{Brion98}.

\subsection{Bruhat order on involutions}\label{bruhat-sect}

Recall that we write $<$ for the Bruhat order on $S_\ZZ$.
Since $\I_\ZZ \subset S_\ZZ$, we can consider the restriction of $<$ to $\I_\ZZ$.
To prove a transition formula analogous to Theorem \ref{thm:ordinary-transition-formula}
for the involution Schubert polynomials $\iS_y$, we need a rather sophisticated understanding of
this restricted partial order, and this section contains a number of technical results for this
purpose.
We are aided by prior work of Incitti \cite{Incitti1}, Hultman \cite{H1,H2,H3}, and Hultman and Vorwerk \cite{H4},
which we recall as follows.

\begin{theorem}[Hultman and Vorwerk \cite{H1,H4}] \label{invbruhat-thm}
The following properties of $(\I_\ZZ, <)$ hold:
\ben
\item[(a)] $(\I_\ZZ,<)$ is a graded poset with rank function $\ellhat$.

\item[(b)] Fix $y,z \in \I_\ZZ$, $(s_{i_1},\dots,s_{i_k}) \in \iR(z)$, and $w \in \cA(z)$.
           The following are then equivalent:
\ben
\item[1.] $y\leq z$.
\item[2.] A subword of the involution word $ (s_{i_1},\dots,s_{i_k}) \in \iR(z)$
          belongs to $\iR(y)$.
\item[3.] An atom $v \in \cA(y)$ exists such that $v \leq w$.
\een
\een
\end{theorem}

\begin{proof}
Part (a) follows by \cite[Theorems 4.3 and 4.8]{H1}.
The equivalence of 1 and 2 in part (b) is stated as \cite[Proposition 2.5]{H4}.
The equivalence of 2 and 3 in part (b) is clear in view of the well-known subword
characterization of the Bruhat order.
\end{proof}

Recall from the introduction that we write $y\lessdot_\I z$ if $z \in \I_\ZZ$ covers $y \in \I_\ZZ$ in the partial order
given by restricting $<$ to $\I_\ZZ$.
Note that while $y \lessdot_\I z $ $\Rightarrow$ $y<z$ and
$y \lessdot z$ $\Rightarrow $ $y \lessdot_\I z$, it does not hold that
$y \lessdot_\I z$ $\Rightarrow$ $y \lessdot z$ for $y,z \in \I_\ZZ$.
The  preceding theorem implies the following:

\begin{corollary}\label{ct-cor}
Let $y,z \in \I_\ZZ$ and $w \in \cA(z)$.
Then $y\lessdot_\I z$ if and only if there exists $v \in \cA(y)$ and
a transposition $t \in S_\ZZ$ such that $v\lessdot vt=w$.
\end{corollary}

In $S_\ZZ$, it is straightforward to associate a transposition to each Bruhat covering relation
$y\lessdot z$; namely, the associated transposition $t$ is the unique one such that $t=y^{-1}z$.
To do something like this for Bruhat covers in $\I_\ZZ$, we need a  stronger form of the preceding
corollary. Specifically, we need to show that if $y \in \I_\ZZ$  and $t=(i,j)$ are fixed, then at
most one involution $z \in \I_\ZZ$ exists such that $w\lessdot wt\in \cA(z)$  for any atoms
$w \in \cA(y)$. If this were guaranteed, then it would  be natural to label the Bruhat cover
$y \lessdot_\I z$ by $t$.
This property is precisely Theorem~\ref{tauintro-thm} from the introduction.
To motivate our proof of this result, we begin by examining some
instructive examples.

\begin{example}\label{ourex}
Consider the involutions $y\in S_n$  such that $[n] = \{a,b,y(a),y(b)\}$ for numbers
$a < b $ in $ [n]$.
There are two such involutions for $n=2$, three for $n=3$, three for $n=4$,
and none for all other values of $n$; in cycle notation, these are given by
\be
\label{2orb}
(1)(2) , \quad
(1,2), \quad
(1,2)(3) , \quad
(1)(2,3) , \quad
(1,3)(2) , \quad
(1,2)(3,4) , \quad
(1,3)(2,4) , \quad
(1,4)(2,3).
\ee
Let $y \in \I_n$ be one of these involutions and suppose $a<b$ are such that
$[n] =\{a,b,y(a),y(b)\}$.
At most one  $z \in \I_n$ exists with $\ellhat(z) = \ellhat(y)+1$ and
$\{ w(a,b) : w \in \cA(y)\} \cap \cA(z) \neq \varnothing$.
Define $\tau_{ab}(y)$ to be this involution $z$ when it exists, and otherwise set
$\tau_{ab}(y) = y$.
We compute that
\[
\ba
\tau_{ab} \Bigl(\arcstart
{
*{.}     & *{.}
}
\arcstop \Bigr)
&=
\arcstart
{
*{.}   \arc{.6}{r}  & *{.}
}
\arcstop
&&\text{for }(a,b) =(1,2)
\\
\tau_{ab}  \Bigl( \arcstart
{
*{.}    \arc{.6}{r}  & *{.} & *{.}
}
\arcstop\Bigr)
&=
 \arcstart
{
*{.}    \arc{.8}{rr}  & *{.} & *{.}
}
\arcstop
&&\text{for }(a,b) \in \{ (2,3), (1,3)\}
\\
\tau_{ab}  \Bigl( \arcstart
{
*{.}     & *{.}  \arc{.6}{r} & *{.}
}
\arcstop\Bigr)
&=
 \arcstart
{
*{.}    \arc{.8}{rr}  & *{.} & *{.}
}
\arcstop
&&\text{for }(a,b) \in \{ (1,2), (1,3)\}
\\
\tau_{ab}  \Bigl( \arcstart
{
*{.}  \arc{.6}{r}   & *{.}    & *{.} \arc{.6}{r}& *{.}
}
\arcstop\Bigr)
&=
\arcstart
{
*{.}  \arc{.8}{rr}   & *{.} \arc{.8}{rr}   & *{.} & *{.}
}
\arcstop
&&\text{for }(a,b) =(2,3)
\\
\tau_{ab}  \Bigl( \arcstart
{
*{.}  \arc{.6}{r}   & *{.}    & *{.} \arc{.6}{r}& *{.}
}
\arcstop\Bigr)
&=
\arcstart
{
*{.}  \arc{.8}{rrr}   & *{.}    & *{.} & *{.}
}
\arcstop
&&\text{for }(a,b) \in \{ (1,3),(2,4),(1,4)\}
\\
\tau_{ab} \Bigl( \arcstart
{
*{.}  \arc{.8}{rr}   & *{.} \arc{.8}{rr}   & *{.} & *{.}
}
\arcstop\Bigr)
&=
\arcstart
{
*{.}  \arc{.8}{rrr}   & *{.}  \arc{.4}{r}  & *{.} & *{.}
}
\arcstop
&&\text{for }(a,b) \in \{ (1,2), (3,4), (1,4)\}
\ea
\]
and that $\tau_{ab}(y)=y$ for all other choices of $y\in \I_n$ and $a<b$ with
$[n] = \{a,b,y(a),y(b)\}$.
In all of these cases,
we have $y \leq \tau_{ab}(y)$.
It also nearly holds that $\tau_{ab}(y) \neq y$ if and only if $y\lessdot y(a,b)$,
with one exception:
$y(a,b) < y$ but $\tau_{ab}(y) \neq y$ for $y = 3412=(1,3)(2,4)$ and $(a,b) =(1,4)$.
\end{example}

Let $E\subset \ZZ$ be a finite set of size $n$, and
write $\phi_E$ and $\psi_E$ for the unique order-preserving bijections
$\phi_E : [n] \to E$ and $\psi_E : E \to [n]$.
Given $w \in S_\ZZ$, we  define
\be\label{standard-eq}
[w]_E = \psi_{w(E)} \circ w \circ \phi_E \in S_n\subset S_\ZZ.
\ee
The operation $[\cdot]_E$ is a special case of the \emph{flattening map}
introduced in the literature on pattern avoidance, which may be defined from any Coxeter group (in our case, $S_\ZZ$) to one of its parabolic subgroups (in our case, the copy of $S_n$ given by the permutations of $\ZZ$ fixing $\ZZ\setminus E$ pointwise). For more background, see the general discussion in \cite[\S5]{AbeBilley} or \cite[\S2]{BilleyBraden}.
The flattened permutation $[w]_E$ is also sometimes called the \emph{standardization} of $w$ with respect to $E$.
Our notation is intended to distinguish $[w]_E $ from the restriction of $w$ to $E$,
which we instead denote as
$
w|_E : E \to \ZZ.
$
We quote without proof some elementary properties of these operations:

\begin{lemma}\label{std0-lem}
If $y,z \in S_\ZZ$ and $E=z(E)\subset \ZZ$ is finite, then
$[yz]_E = [y]_E [z]_E$ and  $[z^{-1}]_E = [z]_E^{-1}$.
\end{lemma}

We denote the \emph{support} of $w \in S_\ZZ$ by $\supp(w) = \{ i \in \ZZ : w(i) \neq i\}$.
\begin{corollary}\label{std2-cor}
Fix $w \in S_\ZZ$ and let $E\subset \ZZ$ be a finite set.
\ben
\item[(a)] If $\supp(w) \subset E$,  then $w=1$ if and only if $[w]_E=1$.
\item[(b)] If $w \in \I_\ZZ$ and $w(E)=E$, then $[w]_E \in \I_\ZZ$.
\een
\end{corollary}

%The flattening map interacts nicely with the Bruhat order on $S_\ZZ$ in the following sense.

The following lemma and its corollary are special cases of 
\cite[Theorem 2]{BilleyBraden},
as is explained in detail in \cite[\S2.3]{BilleyBraden}.
We include self-contained, elementary proofs for completeness.

\begin{lemma}\label{std3-lem}
Let $w \in S_\ZZ$ and suppose $E\subset \ZZ$ is a finite set with $w(E) = E$.
If $t = (i,j) \in S_\ZZ$ is a transposition with $\supp(t)=\{i,j\} \subset E$,
then the following properties hold:
\ben
\item[(a)]  $w < wt$ if and only if $[w]_E < [w]_E [t]_E$.
\item[(b)] If $w \lessdot wt $ then $[w]_E \lessdot [w]_E [t]_E$.
\een
\end{lemma}

\begin{proof}
Assume $i<j$, let $a=\psi_E(i)$ and $b= \psi_E(j)$, and define $r = [t]_E=(a,b) \in S_\ZZ$ and $u = [w]_E$.
Part (a) is clear since $u(a) < u(b)$
if and only if $w(i) < w(j)$.
For (b), suppose $w\lessdot wt$. It suffices to show that $\ell(ur) = \ell(u)+1$.
This holds as $\inv(u) = \inv(w) \cap (E\times E)$ and $\inv(ur)= \inv(wt) \cap (E\times E)$
and since $\inv(wt)$ is formed by adding to $ \inv(w) $ the single inversion $(i,j) \in E\times E$.
\end{proof}

\begin{corollary}\label{std3-cor}
Let $y,z \in S_\ZZ$ and suppose $E\subset \ZZ$ is a finite set such that
$y(E) = z(E) = E$ and $ \supp(z^{-1}y)\subset E$.
If it holds that $[y]_E \leq [z]_E$, then $y \leq z$.
\end{corollary}

\begin{proof}
If $[y]_E= [z]_E$ then $y=z$. % since $ \supp(z^{-1}y)\subset E$,
If $[y]_E  < [z]_E$ then some $t=(i,j) \in S_\ZZ$ with $\{i,j\}\subset E$
has $[y]_E< [y]_E[t]_E=[yt]_E  \leq [z]_E$, in which case $y < yt$ by  Lemma \ref{std3-lem}(a) and $yt \leq z$ by induction.
\end{proof}

Example \ref{ourex} defines $\tau_{ab}(y)$ when $y \in \I_n$ and $a<b$ are such that
$[n] = \{a,b,y(a),y(b)\}$. To define $\tau_{ij}(y)$ in general, we bootstrap that definition
as follows:

\begin{definition}\label{tau-def}
For $y \in \I_\ZZ$ and $i<j $ in $\ZZ$, let $A = \{i,j, y(i), y(j)\}$ and $B = \ZZ\setminus A$,
and define  $\tau_{ij}(y)$ as the
unique
permutation $z \in S_\ZZ$  such that
\[
 \tau_{ab}([y]_A) = [z]_A
\qquand
 y|_{B} =  z|_{B},
\]
 where $a = \psi_A(i)$ and $b= \psi_A(j)$, and $\tau_{ab}([y]_A)$ is defined as in
 Example \ref{ourex}.
\end{definition}

Note that $\tau_{ij}(\tau_{ij}(y)) = \tau_{ij}(y)$ for all integers $i<j$ and $y \in \I_\ZZ$.

\begin{remark}
Our definition of $\tau_{ij}(y)$  is almost the same as Incitti's definition of $\ct_{ij}(y)$ in
\cite{Incitti1}. The differences are as follows. Let $y \in \I_\ZZ$ and $i<j$ in $\ZZ$.
Incitti only defines $\ct_{ij}(y)$ in the case when $y \lessdot y(i,j)$ and either $i < y(i)$, or
$i = y(i)$ and $j \leq y(j)$. If these conditions hold, then $y(i) < y(j)$ and Incitti's definition
(cf. \cite[Table 1]{Incitti1})
becomes
\be\label{ct-def}
\ct_{ij}(y) = \tau_{y(i),y(j)}(y).
\ee
Theorem \ref{tau-thm} will give some justification for our differing conventions.
\end{remark}

Table \ref{ct-fig} makes Definition \ref{tau-def} more explicit, but our initial formulation
captures the main idea. Given integers $i<j$, we define $\tau_{ij}(y)$ as a permutation differing
from $y$ only in its action on a union of two of its cycles, that is, on at most four integers.
By construction, 
if $E\subset \ZZ$ is a finite set such that $\{i,j\}\subset E=y(E)$, then 
$[\tau_{ij}(y)]_E = \tau_{ab}([y]_E)$ for $a = \psi_E(i)$ and $b=\psi_E(j)$.

\begin{example}
If $y \in \I_{11}$ is given by
\[
y = (1,10)(2,5)(4,8)(6,11) =
\arcstart
{
*{.}  \arc{1.6}{rrrrrrrrr}   & *{.}  \arc{.6}{rrr}
&
*{.}    & *{.} \arc{.8}{rrrr}
&
*{.}    & *{.} \arc{1.0}{rrrrr}
&
*{.}   & *{.} 
&
*{.}     & *{.} 
&
*{.}
}
\arcstop
\]
and $(i,j)=(2,11)$ then
\[
\tau_{ij}(y) =
 (1,10)(2,11)(4,8) =
\arcstart
{
*{.}  \arc{1.6}{rrrrrrrrr}   & *{.}  \arc{1.6}{rrrrrrrrr}
&
*{.}    & *{.} \arc{.8}{rrrr}
&
*{.}    & *{.}
&
*{.}   & *{.}
&
*{.}     & *{.}
&
*{.}
}
\arcstop
\]
\end{example}

A few useful properties of $\tau_{ij}(y)$ are essentially 
trivial consequences of its definition.

\begin{lemma}\label{inspection-lem}
Let $y\in \I_\ZZ$ and $i<j$ be in $\ZZ$, and define $A = \{i,j,y(i),y(j)\}$.
\ben

\item[(a)] If $\tau_{ij}(y) \neq y$ then $\cA([y]_A)$ has exactly one element.

\item[(b)] If $\tau_{ij}(y) = \tau_{kl}(y) \neq y$ for some $k<l$ in $\ZZ$ then
           $k \in \{i,y(i)\}$ and $l \in \{j,y(j)\}$.

\een

\end{lemma}

\begin{proof}
These statements hold by construction or  by  inspecting the data in Table \ref{ct-fig}.
%
%By definition $\tau_{ij}(y) \neq y$ precisely when
%$[y]_A \in \{ 1, (1,2), (2,3), (1,2)(3,4), (13)(2,4)\}$;
%to prove (a), one can check directly that each  involution in this set has a single atom,
%or deduce this from our general result \cite[Corollary 6.11]{HMP2}.
%The claim in (b) is evident from the third column in Table \ref{ct-fig}.
%Part (c) holds since by definition $[\tau_{ij}(y)]_A = \tau_{ab}([y]_A)$ and since,
%by an easy calculation, we have $z \leq \tau_{ab}(z)$ whenever $z \in \I_n$ and $a<b$
%are such that $[n] = \{a,b,z(a),z(b)\}$, as was observed in Example \ref{ourex}.
%Finally, (d) follows
%as a straightforward exercise; we omit the details.
\end{proof}

\begin{table}[h]
\[
\barr{| c | c | c | c | c | l}
\hline&&&&\\
A=\{i,j,y(i),y(j)\} & [y]_A & (i,j) & [\tau_{ij}(y)]_A & \sigma\text{ such that }\tau_{ij}(y)=y\sigma
\\&&&&\\
\hline
&&&&\\
\{a<b\}
&
\arcstart
{
*{.}     & *{.}
}
\arcstop
&(a,b)&
\arcstart
{
*{.}   \arc{.6}{r}  & *{.}
}
\arcstop
&
(a,b)
\\&&&&\\
\hline
&&&&\\
\{a<b<c\}
&
 \arcstart
{
*{.}    \arc{.6}{r}  & *{.} & *{.}
}
\arcstop
&
(b,c), (a,c)
&
 \arcstart
{
*{.}    \arc{.8}{rr}  & *{.} & *{.}
}
\arcstop
&
(a,c,b)
\\ &&&& \\
&
\arcstart
{
*{.}     & *{.}  \arc{.6}{r} & *{.}
}
\arcstop
&
(a,b),(a,c)
&
 \arcstart
{
*{.}    \arc{.8}{rr}  & *{.} & *{.}
}
\arcstop
&
(a,b,c)
\\&&&&\\
\hline&&&&\\
\{a<b<c<d\}
&
  \arcstart
{
*{.}  \arc{.6}{r}   & *{.}    & *{.} \arc{.6}{r}& *{.}
}
\arcstop
&
(b,c)
&
\arcstart
{
*{.}  \arc{.8}{rr}   & *{.} \arc{.8}{rr}   & *{.} & *{.}
}
\arcstop
&
(a,d)(b,c)
\\ &&&& \\
&
  \arcstart
{
*{.}  \arc{.6}{r}   & *{.}    & *{.} \arc{.6}{r}& *{.} 
}
\arcstop
& (a,c),(b,d),(a,d)
&
\arcstart
{
*{.}  \arc{.8}{rrr}   & *{.}    & *{.} & *{.} 
}
\arcstop
&
(a,c,d,b)
\\ &&&& \\
&
  \arcstart
{
*{.}  \arc{.8}{rr}   & *{.} \arc{.8}{rr}   & *{.} & *{.} 
}
\arcstop
& (a,b),(c,d) , (a,d)
&
\arcstart
{
*{.}  \arc{.8}{rrr}   & *{.}  \arc{.4}{r}  & *{.} & *{.} 
}
\arcstop
&
(a,b)(c,d)
\\
&&&&\\\hline
\earr
\]
\caption{Values of $\tau_{ij}(y)$.
Fix $y \in \I_\ZZ$ and $i<j$ in $\ZZ$, and define $A = \{i,j,y(i),y(j)\}$.
The first column labels the elements of $A$ in increasing order.
The third column rewrites $(i,j)$ in this labeling.
The second and fourth columns identify the matchings which represent $[y]_A$ and $[\tau_{ij}(y)]_A$. 
For  values of $y$ and $i<j$  not matching  any rows in this table,
we have defined $\tau_{ij}(y) = y$.
}\label{ct-fig}
\end{table}

\begin{proposition}\label{taudef-prop}
Let $y \in \I_\ZZ$ and fix integers $i<j$. Then $y\leq \tau_{ij}(y)$.
\end{proposition}

\begin{proof}
Assume $y \neq \tau_{ij}(y)$ and let $A = \{ i,j,y(i),y(j)\}$. Since $y$ and $\tau_{ij}(y)$
both preserve $A$ and since $\supp(\tau_{ij}(y)^{-1}y) \subset A$, to show that $y<\tau_{ij}(y)$
it suffices by Corollary \ref{std3-lem} to check that $[y]_A < [\tau_{ij}(y)]_A$,
and this holds by inspection from Example~\ref{ourex}.
\end{proof}

\begin{lemma}\label{ct-lem0}
If $y \in \I_\ZZ$ and $t \in S_\ZZ$ is a transposition such that $y\lessdot yt$,
then $y\lessdot yt'$ for $t' = yty$.
\end{lemma}

\begin{proof}
This holds since $u\lessdot v$ if and only if $u^{-1}\lessdot v^{-1}$,
and we have $y=y^{-1}$ and $ (yt)^{-1} = yt'$.
\end{proof}

The motivation for our seemingly \emph{ad hoc}  definition of $\tau_{ij}(y)$, and
the reason why Incitti has defined essentially the same notation in \cite{Incitti1},
is that this construction gives the ``correct'' labeling of the Bruhat covers in $(\I_\ZZ, <)$,
in the following sense.

\begin{theorem}[Incitti \cite{Incitti1}] \label{taubruhat-thm}
Let $y,z \in \I_\ZZ$. The following are then equivalent:
\ben
\item[(a)] $y\lessdot_\I z$.
\item[(b)] $z = \tau_{ij}(y)$ for some $i<j$ in $\ZZ$ and  $\ellhat(z) = \ellhat(y)+1$.
\item[(c)] $z = \tau_{ij}(y)$ for some $i<j$ in $\ZZ$ with $y(i)\leq i$ and $y\lessdot y(i,j)$.
\item[(d)] $z = \tau_{ij}(y)$ for some $i<j$ in $\ZZ$ with $j\leq y(j)$ and $y\lessdot y(i,j)$.
\een
\end{theorem}

\begin{proof}
We attribute this result to Incitti since it is essentially \cite[Theorem 5.1]{Incitti1};
some explanation is required to deduce our particular formulation, however.

Incitti \cite[Theorem 5.1]{Incitti1} proves (cf.\ the remark after Definition \ref{tau-def})
that $y\lessdot_\I z$ if and only if $z = \tau_{y(i),y(j)}(y)$ for some integers $i<j$ with
$y\lessdot y(i,j)$ and either  $i<y(i)$ or $i=y(i)<j\leq y(j)$.
Since $y\lessdot y(i,j)$ if and only if $y\lessdot y(y(i),y(j))$ by Lemma \ref{ct-lem0},
and since  inspecting Table \ref{ct-fig} shows that $\tau_{ij}(y) = \tau_{y(i),y(j)}(y)$
when $\tau_{ij}(y) \neq y$ and $i=y(i)<j$, it follows that (a) $\Leftrightarrow$ (c).
From this equivalence and Theorem \ref{invbruhat-thm}, the implication (a) $\Rightarrow$ (b)
is  immediate, while (b) $\Rightarrow$ (a) holds by Proposition \ref{taudef-prop}.
Thus (a) $\Leftrightarrow$ (b) $\Leftrightarrow$ (c).

The remaining equivalence (a) $\Leftrightarrow$ (d)
can be deduced from  (a) $\Leftrightarrow$ (c) using the following easily checked facts:
if   $w^*$ denotes the permutation $i\mapsto -w(-i)$  for $w \in S_\ZZ$ then
(1) $y \lessdot_\I z$ if and only if $y^* \lessdot_\I z^*$,
(2) $ \tau_{ij}(y)^* = \tau_{-j,-i}(y^*)$, and
(3) $y\lessdot y(i,j)$ if and only if $y^* \lessdot y^*(-j,-i)$.
\end{proof}

The preceding results show that the maps $\tau_{ij} : \I_\ZZ\to\I_\ZZ$ provide an effective way of labeling 
the elements covering an involution $y$ in $(\I_\ZZ,<)$.
For an explanation of whether there is a similarly
reasonable way to label the involutions which $y$ covers, see the remark after Theorem~\ref{tau-thm}.

\begin{corollary}\label{taubruhat-cor}
Let $y \in \I_\ZZ$ and suppose $z=\tau_{ij}(y)$ for some $i<j$ in $\ZZ$.
Assume that $y(i) \leq i$ or $j\leq y(j)$. Then $y\lessdot_\I z$ if and only if $y\lessdot y(i,j)$.
\end{corollary}

\begin{proof}
If $y(i) \leq i$, then  $y\lessdot y(i,j)$ $\Rightarrow$ $y\lessdot_\I z$ by
Theorem \ref{taubruhat-thm}, and to prove the reverse implication it suffices
by Lemma \ref{ct-lem0} and Theorem \ref{taubruhat-thm}
to show that if $y \neq \tau_{ij}(y) = \tau_{kl}(y)$ for some  $k<l$
with $y(k) \leq k$ and $y\lessdot y(k,l)$, then $(k,l) \in \{(i,j), (y(i),y(j))\}$.
This follows by inspecting Table \ref{ct-fig}.
When $j \leq y(j)$, it follows that  $y\lessdot_\I z$ $\Leftrightarrow$ $y\lessdot y(i,j)$
by a symmetric argument.
\end{proof}

\subsection{Transition formulas}
\label{itransition-sect}

For $y \in \I_\ZZ$ define $\Cyc_\ZZ(y) = \{ (i,j) \in \ZZ\times \ZZ : i \leq j = y(i)\}$.
In checking certain properties of the set of atoms $\cA(y)$, we are able to reduce some tedious case
 analyses  to finite computer calculations by means of  the following theorem of
 Can, Joyce, and Wyser. This result is equivalent to \cite[Theorem 2.5]{CJW},
and describes  the elements of $\cA(y)$ completely in terms of $\Cyc_\ZZ(y)$.

\begin{theorem}[Can, Joyce, and Wyser \cite{CJW}] \label{atomSn-thm}
Let $y \in \I(S_\ZZ)$ and $w \in S_\ZZ$. Then $w \in \cA(y)$ if and only if
the following properties hold:
\ben
\item[(i)] If $(a,b) \in \Cyc_\ZZ(y)$ is such that $a<b$ then $w(a,b) \lessdot w$.

\item[(ii)] If $(a,b),(a',b') \in \Cyc_\ZZ(y)$ are such that $a<a'$ and $b<b'$ then $ w(a) < w(b')$.
\een
\end{theorem}

The form of the conditions in this theorem is notably ``local'' in the following sense.

\begin{corollary}\label{atom-cor}
If $y \in \I_\ZZ$ and $w \in S_\ZZ$ then the following are equivalent:
\ben
\item[(a)] $w \in \cA(y)$.
\item[(b)] $[w]_E \in \cA([y]_E)$ for all $y$-invariant subsets $E\subset \ZZ$.
\item[(c)] $[w]_E \in \cA([y]_E)$ for all $y$-invariant subsets $E\subset \ZZ$
           containing at most two $y$-orbits.
\een
\end{corollary}

\begin{remark}
Note that if $w \in \cA(y)$ and $E$ is $y$-invariant then it still may happen that $w(E) \neq E$.
\end{remark}

\begin{proof}
We have (a) $\Rightarrow$ (b) $\Rightarrow$ (c) by Lemma \ref{std3-lem} and
Theorem \ref{atomSn-thm}.
It is clear that if (c) holds then condition (ii) in
Theorem \ref{atomSn-thm} holds for $w$. To prove that (c) $\Rightarrow$ (a), we
check that (c) implies that $w(a,b) \lessdot w$ for all $(a,b) \in \Cyc_\ZZ(y)$.
Arguing by contradiction,   suppose (c) holds but condition (i) in Theorem \ref{atomSn-thm}
fails for $(a,b) \in \Cyc_\ZZ(y)$ with $a<b$.  We cannot have $w(a) < w(b)$ since
$[w]_E \in \cA([y]_E)$ for $E = \{a,b\}$,
so some $e \in \ZZ$ has  $a<e<b$ and $w(b)<w(e)<w(a)$.
But then $[w]_E$ fails to be an atom for $[y]_E$ when $E=\{a,b,e,y(e)\}$.
Hence $w(a,b) \lessdot w$, and  (c) $\Rightarrow$ (a).
\end{proof}

The bulk of this section is spent proving two technical theorems about the operator $\tau_{ij}$
introduced in the previous section. Our first  result of this kind is the following:

\begin{theorem}\label{tau-thm}
Let $y \in \I_\ZZ$ and $w \in \cA(y)$.
Suppose $i<j$ in $\ZZ$ are such that $w\lessdot w (i,j)$.
\ben
\item[(a)] If $w(i,j) \in \cA(z)$ for some $z \in \I_\ZZ$, then $z=  \tau_{ij}(y)\neq y$.

\item[(b)] If $w(i,j) \notin \cA(z)$ for all $z \in \I_\ZZ$, then $\tau_{ij}(y)=y$.
\een
\end{theorem}

\begin{remark}
It is not possible to define, for all integers $i<j$, an ``inverse'' map $\eta_{ij} : \I_\ZZ \to \I_\ZZ$
 such that $\eta_{ij}(z)=y$ whenever there exists $w \in \cA(z)$ with $w(i,j) \lessdot w$ and $w \in \cA(y)$.
A map with this property would satisfy $\eta_{ij}(\tau_{ij}(y)) = y$, but $ \tau_{1,4}(321) = \tau_{1,4}(2143) = \tau_{1,4}(1432) = 4231$, for example.
\end{remark}

\begin{proof}
Let $t=(i,j) \in S_\ZZ$.
For both assertions, our strategy will be to show that if a counterexample exists,
then a counterexample exists in some particular finite symmetric group.
We may then confirm each part by checking, via a computer calculation,
that no counterexamples in the relevant finite groups exist. In parsing our argument,
it may be helpful to consult the example given after the proof 
which explains in detail how things work out in a specific case.

For part (a), suppose $wt \in \cA(z)$ for $z \in \I_\ZZ$ with
$z\neq \tau_{ij}(y)$.
Since $y\lessdot_\I z $ by Corollary \ref{ct-cor},
we must have $z = \tau_{pq}(y)$ for some $p<q$ in $\ZZ$ by Theorem \ref{taubruhat-thm}.
Define $E = \{i,j, y(i), y(j), p, q, y(p), y(q)\}$. This set is $y$- and $t$-invariant
by construction, and $z$-invariant since $\supp(zy^{-1}) \subset E$.
By Corollary \ref{atom-cor},  $[w]_E \in \cA([y]_E)$ and $[wt]_E \in \cA([z]_E)$,
while by Lemmas \ref{std0-lem} and \ref{std3-lem}
we have
$[w]_E \lessdot  [w]_E[t]_E= [wt]_E$.
Let $a = \psi_E(i)$ and $b=\psi_E(j)$ so that $a<b$ and $[t]_E = (a,b)$.
It is clear by definition that $y(m)=z(m)=\tau_{pq}(y)(m) = \tau_{ij}(y)(m)$ for all
$m \in \ZZ\setminus E$
and so we deduce by Corollary \ref{std2-cor} 
that $[z]_E \neq [\tau_{ij}(y)]_E = \tau_{ab}([y]_E)$.
As the set $E$ has at most eight elements, these observations show that
if  there exist $y \in \I_\ZZ$, $t=(i,j) \in S_\ZZ$, and $w \in \cA(y)$ with $w\lessdot wt$
contradicting (a), then there exists such a contradiction with $y,z,w \in S_8$.
However, it is a feasible computer calculation to check
that there are no such counterexamples, so (a) holds in general.

For (b), suppose $wt \notin \cA(z)$ for all $z \in \I_\ZZ$ but $\tau_{ij}(y) \neq y$.
Let $u = \tau_{ij}(y)$ and $A = \{i,j,y(i),y(j)\}$. By Corollary \ref{atom-cor}, there
exists a $u$-invariant set $B\subset \ZZ$ with
at most two $u$-orbits such that $[wt]_B \notin \cA([u]_B)$.
It is evident from Table \ref{ct-fig} that $|B \cup y(B)| \leq |B|+2$,
and it must hold that $A\cap B \neq \varnothing$ since if this intersection were empty then
we would have the contradiction $y(B) =B$ and
$[wt]_B = [w]_B \in \cA([y]_B) = \cA([u]_B)$.
Thus the set $E=A\cup B \cup y(B)$  has size at most 9. Clearly  $y(E) = t(E)=E$,
we have $u(E) = E$ since $\supp(\sigma)  \subset A$ for $\sigma = u^{-1}y=uy$, and by construction $[wt]_E \notin \cA([u]_E)$.
We now claim that $[w]_E \in \cA([y]_E)$ and  $[w]_E \lessdot [w]_E [t]_E =[wt]_E \notin \cA(z)$
for all $z \in \I_\ZZ$ but $\tau_{cd}([y]_E) = [\tau_{ij}(y)]_E=[u]_E \neq [y]_E$ for
$c= \psi_E(i)$ and $d = \psi_E(j)$.
Most of this follows exactly as in the previous paragraph using the invariance of $E$ and
various auxiliary results; the principal thing to show is the middle assertion that
$[w]_E[t]_E \notin \cA(z)$ for all $z \in \I_\ZZ$.
This holds since by part (a) we can only have $[w]_E[t]_E=[wt]_E \in \cA(z)$ for some
$z \in \I_\ZZ$ if $z=\tau_{cd}([y]_E)=[u]_E $, but  we have already seen that
$[wt]_E \notin\cA([u]_E)$.
Thus, if  there exist $y \in \I_\ZZ$ and $t=(i,j) \in S_\ZZ$ and $w \in \cA(y)$ with
$w\lessdot wt$ contradicting (b), then there exists such a contradiction with
$y,w \in S_9$. We again confirm by a computer calculation
that there are no such counterexamples, so (b)  holds.
\end{proof}

\begin{example}\label{3.20-ex}
Consider the involution 
\[ y = (1,9)(3,8)(5,10)(6,7) = 
\arcstart
{
*{.}  \arc{1.5}{rrrrrrrr}   &
*{.}  &
*{.}  \arc{1.0}{rrrrr}  &
*{.}  &
*{.} \arc{1.0}{rrrrr} &
*{.} \arc{0.4}{r} &
*{.}   &
*{.}  &
*{.} &
*{.}
}
\arcstop
\]
which has $w = 2,3,5,6,8,10,9,4,1,7  \in \cA(y)$ as an atom. Define $t=(i,j) = (8,10) \in S_\ZZ$.
Then $w\lessdot wt$ and it happens that $wt \in \cA(z)$ for a unique element $z \in \I_\ZZ$.
We can deduce that $z =  \tau_{ij}(y) = (1,9)(3,10)(5,8)(6,7)$
using only information about the atoms of involutions in the finite group $S_8$ in the following way.
We know that $z = \tau_{pq}(y)$ for some integers $p<q$, so suppose these integers are such that
$\tau_{pq}(y)\neq \tau_{ij}(y)$. For example, take $p=4$ and $q=6$ so that 
\[ 
z = (1,9)(3,8)(4,7)(5,10) = 
\arcstart
{
*{.}  \arc{1.5}{rrrrrrrr}   &
*{.}  &
*{.}  \arc{1.0}{rrrrr}  &
*{.} \arc{0.6}{rrr}  &
*{.} \arc{1.0}{rrrrr} &
*{.} &
*{.}   &
*{.}  &
*{.} &
*{.}
}
\arcstop.
\]
Then $E = \{i,j, y(i), y(j), p, q, y(p), y(q)\} = \{3,4,5,6,7,8,10\}$
and  we have 
\[[y]_E =(1,6)(3,7)(4,5)
=
\arcstart
{
*{.}  \arc{1.2}{rrrrr}   &
*{.}  &
*{.}  \arc{1.2}{rrrr}  &
*{.} \arc{0.4}{r}  &
*{.} &
*{.} &
*{.}  
}
\arcstop,
\qquad [z]_E = (1,6)(2,5)(3,7)
=
\arcstart
{
*{.}  \arc{1.2}{rrrrr}   &
*{.}  \arc{0.8}{rrr} &
*{.}  \arc{1.2}{rrrr}  &
*{.} &
*{.} &
*{.} &
*{.}  
}
\arcstop,
\]
$[t]_E = (6,7)$, and 
$[w]_E = {2,3,5,7,6,1,4} \lessdot [w]_E[t]_E=[wt]_E = {2,3,5,7,6,4,1} \in S_8$.
By Corollary~\ref{atom-cor} the last element must belong to $\cA([z]_E)$,
which is a contradiction since one can compute that 
 $[wt]_E$ is an atom of
$[\tau_{8,10}(y)]_E = \tau_{6,7}([y]_E) =(1,7)(3,6)(4,5)\neq [z]_E$.
We reach a similar contradiction for any other integers $p<q$ with $\tau_{pq}(y) \neq \tau_{ij}(y)$,
so we must have $z = \tau_{ij}(y)$.
Crucially, each of these contradictions only depends on calculations involving permutations in $S_8$.
Our proof of Theorem~\ref{tau-thm}(b) reduces similarly to a finite calculation involving just the involutions in $S_9$. 
\end{example}

We recall a useful observation in the proof of Theorem \ref{tau-thm}.

\begin{lemma}\label{reduce-lem}
Let $y \in \I_\ZZ$ and $w \in \cA(y)$. Suppose $t \in S_\ZZ$ is a transposition  with
$w \lessdot wt \notin \cA(z)$ for all $z \in \I_\ZZ$.
If $E\subset \ZZ$ is a $y$-invariant finite set containing  $\supp(t)$,
then  $[w]_E \in \cA([y]_E)$ and $[w]_E \lessdot [w]_E[t]_E  = [wt]_E \notin \cA(z)$
for all $z \in \I_\ZZ$.
\end{lemma}

\begin{proof}
We have  $[w]_E \lessdot [w]_E[t]_E  = [wt]_E$ by Lemmas \ref{std0-lem} and \ref{std3-lem} and
$[w]_E \in \cA([y]_E)$ by Corollary \ref{atom-cor}.
Let $i<j$ be such that $t=(i,j)$. 
The second part of Theorem \ref{tau-thm} implies that $\tau_{ij}(y) = y$, so
$\tau_{ab}([y]_E) = [\tau_{ij}(y)]_E = [y]_E$ for $a =\psi_E(i)$ and $b=\psi_E(j)$.
Hence, by the first part of Theorem \ref{tau-thm} we must have $[wt]_E \notin \cA(z)$
for all $z \in \I_\ZZ$.
\end{proof}

We now have our second technical theorem.

\begin{theorem}\label{mir2-thm}
Fix $y \in \I_\ZZ$, and suppose $i<j$ in $\ZZ$ and $w \in \cA(y)$ are such that
$w \lessdot w(i,j) \notin \cA(z)$ for all $z \in \I_\ZZ$.
The following then holds:
\ben
\item[(a)] There are  unique numbers $i'<j'$  in $\ZZ$ such that $w\neq w(i,j)(i',j') \in \cA(y)$.

\item[(b)] More specifically,  if $A = \{i,j,y(i),y(j)\}$,
then $[w]_A \in \{231, 312, 2431, 3412, 4213\}$ and the    values of $i' <j'$ in (a)
are as specified in Table \ref{mir2-table}, so that  $i' \in \{j, y(j)\}$ and $j' \in \{i,y(i)\}$.
\een
\end{theorem}

\begin{proof}
By Theorem \ref{atomSn-thm}(a) and Theorem \ref{tau-thm}(b), we have
$(i,j) \notin \Cyc_\ZZ(y)$ and $\tau_{ij}(y) = y$,
so
it follows by inspecting Table \ref{ct-fig} that either $[y]_A = (1,3)$, or $[y]_A = (1,3)(2,4)$
and $[(i,j)]_A = (2,3)$, or $[y]_A =(1,4)(2,3)$.
The second case cannot occur since if $[y]_A = (1,3)(2,4)$ then Theorem \ref{atomSn-thm} implies
that
$[w]_A =2413$ for which $(2,3)$ is an inversion.
If $[y]_A = (1,3)$ then Theorem \ref{atomSn-thm} implies that $[w]_A \in \{231, 312\}$, and if
$[y]_A=(1,4)(2,3)$ then it likewise follows that $[w]_A \in \{2431, 3412, 4213\}$.
This confirms the first assertion in part (b), and thus Table \ref{mir2-table} describes all
possible values of $[y]_A$, $[w]_A$, and $(i,j)$.

Suppose
$i'<j'$ are as specified in Table \ref{mir2-table}; note that $\{i',j'\} \subset A$.
It remains to check that (1) $ w(i,j)(i',j') \in \cA(y)$,
and (2) $w(i,j)(k,l) \notin \cA(y)$ for all transpositions $(k,l) \notin \{(i,j),(i',j')\}$.
Our strategy is similar to the one used to prove Theorem \ref{tau-thm}.
If  (1) fails and $ w(i,j)(i',j') \notin \cA(y)$, then  $[w(i,j)(i',j')]_B \notin  \cA([y]_B)$
for some $y$-invariant set $B\subset \ZZ$ with size at most 4 by Corollary \ref{atom-cor},
so by applying Lemma \ref{reduce-lem} with $E=A\cup B$, we may assume without loss of generality
that $y,w,(i,j) \in S_8$.
Likewise, if (2) fails and there exists a transposition $(k,l) \notin \{(i,j),(i',j')\}$ with
$w(i,j)(k,l) \in \cA(y)$, then
$[w(i,j)(k,l)]_E \in \cA([y]_E)$ for  $E=A\cup \{ k,l,y(k),y(l)\}$ by Corollary \ref{atom-cor},
and
by applying Lemma \ref{reduce-lem} we may again assume that
$y,w,(i,j) \in S_8$.  However,  a  computer 
search shows that no counterexamples to (1) or (2) exist in $S_8$, so these claims hold in general.
\end{proof}

\begin{example}
A concrete example is useful for understanding the preceding theorem. Suppose 
\[ y = (1,9)(3,8)(5,10)(6,7) = 
\arcstart
{
*{.}  \arc{1.5}{rrrrrrrr}   &
*{.}  &
*{.}  \arc{1.0}{rrrrr}  &
*{.}  &
*{.} \arc{1.0}{rrrrr} &
*{.} \arc{0.4}{r} &
*{.}   &
*{.}  &
*{.} &
*{.}
}
\arcstop
\]
and $w = 2,3,5,6,8,10,9,4,1,7  \in \cA(y)$ as in Example~\ref{3.20-ex}.
If $i=5$ and $j=6$ then it happens that $w\lessdot w(i,j) \notin \cA(z)$ for all $z \in \I_\ZZ$.
We have $A = \{5,6,7,10\}$ and $[w]_A = 2431$, and
Theorem~\ref{mir2-thm} asserts that 
$7=i'<j'=10$ are the unique integers such that
$w\neq w(i,j)(i',j') \in \cA(y)$.

\end{example}

\begin{table}[h]
\[
\barr{| c | c | c | c | c | cc |}
\hline&&&&&&\\
A=\{i,j,y(i),y(j)\} & [y]_A & [w]_A & (i,j)  & (i',j') & i' & j'
\\&&&&&&\\
\hline
&&&&&&\\
\{a<b<c\}
&
 \arcstart
{
*{.}    \arc{.8}{rr}  & *{.} & *{.}
}
\arcstop
&
231
&
(a,b)
&
(b,c) & j & y(i)
\\
&
&
312
&
(b,c)
&
(a,b) & y(j) & i
\\&&&&&&\\
\hline&&&&&&\\
\{a<b<c<d\}
&
\arcstart
{
*{.}  \arc{.8}{rrr}   & *{.}  \arc{.4}{r}  & *{.} & *{.}
}
\arcstop
&2431
&
(a,b)
&
(c,d) & y(j) & y(i)
\\
&
&
&
(a,c)
&
(c,d) & j & y(i)
\\
&&&&&&
\\
&
&3412
&
(a,b)
&
(b,d) & j &  y(i)
\\
&
&
&
(c,d)
&
(a,c) & y(j) & i
\\
&&&&&&
\\
&
&4213
&
(b,d)
&
(a,b) & y(j) &  i
\\
&
&
&
(c,d)
&
(a,b) & y(j) & y(i)
\\
&&&&&&\\\hline
\earr
\]
\caption{Values of $i'<j'$ in Theorem \ref{mir2-thm}.
Fix $y \in \I_\ZZ$, $w \in \cA(y)$, and $i<j$ in $\ZZ$ such that
$w\lessdot w(i,j) \notin \cA(z)$ for all $z \in \I_\ZZ$.
Let $A = \{i,j,y(i),y(j)\}$. Then $[y]_A$ and  $[w]_A$  and $i<j$
must  correspond to one of the  rows  in the table.
The first column indicates a labeling of the elements of $A$  in increasing order.
The fourth column rewrites  $(i,j)$ in this labeling, and the remaining columns describe
the unique pair $(i',j')$ with $i'<j'$ and $w\neq w(i,j)(i',j') \in \cA(y)$.
}
\label{mir2-table}
\end{table}

We now have a sufficiently detailed understanding of the covering relations in $(\I_\ZZ,<)$
to prove a transition formula for the  polynomials $\iS_y$.
Define $\cT(y,z)$ for  $y,z \in \I_\ZZ$ as the set of transpositions
\[
\cT(y,z) = \{  (i,j)  \in S_\ZZ : \exists w\in\cA(y)\text{ with }w\lessdot w(i,j) \in \cA(z)\}.
\]
By Corollary \ref{ct-cor}, $\cT(y,z)$ is nonempty if and only if $y \lessdot_\I z$.
Moreover, by Theorem \ref{tau-thm}(a), if $i<j$ are integers such that $(i,j) \in \cT(y,z)$ then
$z = \tau_{ij}(y)$.

\begin{corollary}\label{T-cor}
Let $x,y,z \in \I_\ZZ$ with $y\neq z$. Then
$\cT(x,y) \cap \cT(x,z) = \varnothing$.
\end{corollary}

\begin{proof}
If $i<j$  are such that
 $(i,j) \in \cT(x,y)\cap \cT(x,z)$ then   $y=z = \tau_{ij}(x)$ by Theorem \ref{tau-thm}.
\end{proof}

\begin{lemma}\label{tech-lem1}
Let $y,z \in \I_\ZZ$ with $y\lessdot_\I z$. Suppose $t_1,t_2 \in \cT(y,z)$ and
$w_1,w_2 \in \cA(y)$ are such that $w_i \lessdot w_it_i $, so that  $w_it_i \in \cA(z)$,
for each $i \in \{1,2\}$. If $w_1t_1=w_2t_2$ then $w_1=w_2$ and $t_1=t_2$.
\end{lemma}

\begin{proof}
Suppose $w_1t_1=w_2t_2$, so that $w_2 = w_1t_1t_2 \in \cA(y)$. It suffices to show that $t_1t_2=1$.
Let $i<j$ be such that $t_1 = (i,j)$ and define $A = \{i,j,y(i),y(j)\}$.
We have $z = \tau_{ij}(y)$ by Theorem \ref{tau-thm}(a), so it follows from
Lemma \ref{inspection-lem} that  $\cA([y]_A)$ has exactly one element
and $\supp(t_2) \subset A$.
One can deduce that $[w_1]_A$ and $[w_2]_A$ each belong to the singleton set $ \cA([y]_A)$
by looking at Table~\ref{ct-fig} and using Theorem~\ref{atomSn-thm}.
We therefore have $[w_1]_A = [w_2]_A = [w_1t_1t_2]_A = [w_1]_A[t_1t_2]_A$ by Lemma \ref{std0-lem}.
Thus $[t_1t_2]_A = 1$, so $t_1t_2=1$ by Corollary \ref{std2-cor}.
\end{proof}

For any  subset   $\cT\subset S_\ZZ$ and $y \in \I_\ZZ$, define
$\cA(y;\cT) = \{ (w,t)  \in  \cA(y) \times \cT: w \lessdot wt\}$.
Of course, we are only interested in this set when $\cT$ is composed of transpositions.

\begin{corollary}\label{tech-cor1}
If $y\lessdot_\I z$  and $\cT=\cT(y,z)$ then
$(w,t) \mapsto wt$ is  a bijection
$ \cA  (y; \cT) \to \cA(z)$.
\end{corollary}

\begin{proof}
The given map is surjective by Corollary \ref{ct-cor} and injective by Lemma \ref{tech-lem1}.
\end{proof}

Fix $y \in \I_\ZZ$ and $r \in \ZZ$, and recall the definitions of the sets $\hat\Phi^\pm(y,r)$ from \eqref{hatphi-eq}.
Note by Proposition~\ref{taudef-prop}
that if $z = \tau_{rj}(y)$ or $z=\tau_{ir}(y)$ where $i<r<j$, then $y\lessdot_\I z$ if and only if $\ellhat(z) = \ellhat(y)+1$.
Corollary \ref{taubruhat-cor} implies, in turn,  that:
\begin{itemize}
\item If $y(r) \leq r$  then
$z \in \hat\Phi^+(y,r)$ if and only if $z = \tau_{rj}(y)$ and $y \lessdot y(r,j)$ for some $j>r$.

\item If  $r\leq y(r)$ then
$z \in \hat\Phi^-(y,r)$ if and only if $z=\tau_{ir}(y)$ and $y\lessdot y(i,r)$ for some $i<r$
\end{itemize}
We note a few other straightforward properties of these sets.

\begin{proposition}\label{nonempty-prop}
If $y \in \I_\ZZ$ and $(p,q) \in \Cyc_\ZZ(y)$ then 
$\hat\Phi^-(y,p)$ and $\hat\Phi^+(y,q)$ are both nonempty.
\end{proposition}

\begin{proof}
Let  $i$ and $j$ be  respectively maximal and minimal such that $i < p \leq q < j$
and $y(i) < y(p)$ and $y(q) < y(j)$. Then  $y \lessdot y(i,p)$ and  $y \lessdot y(q,j)$
so $\tau_{ip}(y) \in \hat\Phi^-(y,p)$ and $\tau_{qj}(y) \in \hat\Phi^+(y,q)$.
\end{proof}

\begin{lemma}\label{phi+lem}
If $y \in \I_\ZZ$ and $(p,q) \in \Cyc_\ZZ(y)$
then
$\hat\Phi^+(y,p) \subset \hat\Phi^+(y,q) $ and $ \hat\Phi^-(y,q) \subset \hat\Phi^-(y,p) .$
\end{lemma}

\begin{proof}
Fix $(p,q) \in \Cyc_\ZZ(y)$ and suppose $z \in \hat\Phi^+(y,p)$, so that 
   $y \lessdot_\cI z$ and $z=\tau_{pj}(y)$ for some integer $j>p$.
The properties of $\tau_{ij}$ indicated in Table~\ref{ct-fig} imply the following statements:
\begin{itemize}
\item If $p=q$ then $\hat\Phi^+(y,p)= \hat\Phi^+(y,q)$.
\item If $p<q$ then $j=y(j)$ then we must have $q<j$ and $z = \tau_{qj}(y) \in \hat\Phi^+(y,q)$.
\item If $p<q$ and $j<y(j)$ then we must have $q<y(j)$ and $z=\tau_{q,z(j)}(y) \in \hat \Phi^+(y,q)$.
\item If $p<q$ and $y(j)<j$ then we must have $q<j$ and $z=\tau_{qj}(y) \in \hat \Phi^+(y,q)$.
\end{itemize}
We conclude that $\hat\Phi^+(y,p) \subset \hat\Phi^+(y,q) $.
The argument that $\hat\Phi^-(y,q)  \subset \hat\Phi^-(y,p) $ is similar.
\end{proof}

Let $ \Cyc_\PP(y) = \Cyc_\ZZ(y)\cap (\PP\times \PP)$ for $y \in \I_\infty$,
and set
$x_{(p,q)} = 2^{-\delta_{pq}} (x_p+x_q)$ for $(p,q) \in \PP\times \PP$.
We may now give the proof of Theorem~\ref{thm:involution-transition-formula}, which we restate here for convenience:

\begin{theorem}[Restatement of Theorem \ref{thm:involution-transition-formula}]
\label{invmonk-thm}
If $y \in \I_\infty$ and $(p,q) \in \Cyc_\PP(y)$
then
\[
 x_{(p,q)} \iS_y = \sum_{ z \in \hat\Phi^+(y,q)} \iS_{z} -  \sum_{z \in \hat\Phi^-(y,p)} \iS_{z}
\]
 where we set $\iS_z = 0$ for $z \in \I_\ZZ-\I_\infty$.
\end{theorem}

\begin{proof}
Let $\cT^+$ be the set of transpositions $t = (r,j)$ with $r \in \{p,q\}$ and $r<j$.
As $\iS_y = \sum_{w \in \cA(y)} \fkS_w$,  the original transition formula
(Theorem \ref{thm:ordinary-transition-formula}) implies that
\[
 x_{(p,q)} \iS_y =
 \sum_{(w,t) \in \cA(y; \cT^+)} \fkS_{wt} - \sum_{(w,t) \in \cA(y; \cT^-)} \fkS_{wt}
 .
\]
By Lemma \ref{phi+lem}, we have $ \hat \Phi^+(y,p) \cup \hat \Phi^+(y,q)=  \hat\Phi^+(y,q)$; let $\cZ^+$ denote this set.
Lemma \ref{inspection-lem}(b) and
Theorem \ref{tau-thm}(a) imply that 
$\cT(y,z) \subset \cT^+$ for $z \in \cZ^{+}$.
The sets $\cT(y,z)$ are  disjoint  as $z \in \cZ^+$ varies by Corollary \ref{T-cor}, and
by Corollary \ref{tech-cor1}
we have
$ \sum_{(w,t) \in \cA(y; \cT(y,z))} \fkS_{wt} = \iS_z$.
Therefore
\[
\sum_{(w,t) \in \cA(y; \cT^+)} \fkS_{wt} = \sum_{z \in \hat\Phi^+(y,q)} \iS_z + \sum_{(w,t) \in \cA(y; \cN^+)} \fkS_{wt} \]
where $\cN^+ = \cT^+ - \bigcup_{z \in \cZ^+} \cT(y,z)$.
Define 
$\cT^-$ as the set of transpositions $t = (i,r)$ with $r \in \{p,q\}$ and $1\leq i<r$
and let
$\cZ^- =\hat \Phi^-(y,p) \cup \hat \Phi^-(y,q) = \hat\Phi^-(y,p) $.
It follows by the same arguments that 
\[
\sum_{(w,t) \in \cA(y; \cT^-)} \fkS_{wt} = \sum_{z \in \hat\Phi^-(y,p)} \iS_z + \sum_{(w,t) \in \cA(y; \cN^-)} \fkS_{wt} \]
where $\cN^- = \cT^- - \bigcup_{z \in \cZ^-} \cT(y,z)$,
so it suffices to show
$ \sum_{(w,t) \in \cA(y; \cN^+)} \fkS_{wt} =   \sum_{(w,t) \in \cA(y; \cN^-)} \fkS_{wt}$.
This follows directly from  Theorem \ref{mir2-thm}, since
if $(w,t)$ belongs to $\cA(y; \cN^+)$  or $\cA(y; \cN^-)$ then  $w\lessdot wt \notin \cA(z)$
must hold for all $z \in \I_\infty$ by Theorem \ref{tau-thm}(a).
\end{proof}

\begin{example}
\label{ex:transition}
Using Theorem~\ref{invmonk-thm}, we see that
\[
(x_3+x_4)\iS_{15432} = \iS_{156423} - \iS_{45312}
\]
since $\tau_{4,6}(15432) = 156423$ and $\tau_{1,3}(15432) = 45312$ while $\tau_{2,3}(15432) = \tau_{4,5}(15432) = 15432$.
Alternatively, we can compute $(x_3+x_4)\iS_{15432}$ directly from the original transition formula as follows.
It holds that $\cA(15432) = \{13542,14523,15324\}$,
so
\begin{align*}
(x_3+x_4)\iS_{15432} & = (x_3 + x_4)\fkS_{13542} + (x_3 + x_4)\fkS_{14523} + (x_3+x_4)\fkS_{15324}\\
& = (\fkS_{15423} - \fkS_{35124}) + (\fkS_{15342} - \fkS_{25314}) \\
& \ \ \ \  + (\fkS_{146235} - \fkS_{15423}) + (\fkS_{14532} - \fkS_{24513})\\
& \ \ \ \  + (\fkS_{136425} - \fkS_{15342}) + (\fkS_{135624} - \fkS_{14532})\\
& =  \fkS_{135624} + \fkS_{136425} + \fkS_{146235} - \fkS_{25314} - \fkS_{24513} - \fkS_{35124}\\
& =  \iS_{156423} - \iS_{45312}.
\end{align*}
Here, the second equality comes from applying Theorem~\ref{thm:ordinary-transition-formula} to each of the six terms in the first line, while the final equality holds since $\cA(156423) = \{135624,136425,146235\}$ and $\cA(45312) = \{25314,24513,35124\}$.
All of the terms that cancel in the third equality correspond to permutations that are not atoms, as predicted by Theorem~\ref{mir2-thm}.
\end{example}

Recall that $w_n=n\cdots 321$ denotes the longest element in $S_n$, which is an involution.
Theorem \ref{invmonk-thm} leads to an alternate proof of Wyser and Yong's
formula for $\iS_{w_n}$ from \cite{WY}. 

\begin{corollary}[Wyser and Yong \cite{WY}] \label{wy-cor}
Let $n \in \PP$.
Then $\iS_{w_n} =\ds \prod_{\substack{1\leq i \leq j \leq n  \\ i+j \leq n}}  x_{(i,j)}$.
\end{corollary}

Can, Joyce, and Wyser  have further generalized this product formula as \cite[Eq.\ (19)]{CJW2}.

\begin{remark}
Wyser and Yong's approach in \cite{WY}
is to take this formula
as the definition of $\iS_{w_n}$, and then specify $\iS_y$ for $y \in \I_n$ inductively
according to the rule that
$\iS_{y} = \partial_i \iS_{s_i\circ y \circ s_i}$ for $s_i \notin \DesR(y)$.
(In the notation of  \cite{WY}, $\iS_y$ would be denoted
$2^{-\kappa(y)} \Upsilon_{y,(\GL_n,\O_n)}$.)
From this definition, it is a nontrivial result that the polynomials $\iS_y$ for $y \in \I_n$
are well-defined \cite[Theorem 1.1]{WY} and have no dependence on $n$ \cite[Theorem 1.4]{WY}.
From our Definition \ref{iS-def}, conversely, these properties are automatic while the given
formula for $\iS_{w_n}$ is nontrivial. The self-contained  proof below  gives another means
of seeing that these two approaches lead to equivalent definitions of $\iS_y$.
\end{remark}

\begin{proof}
We may assume that $n>1$.
Let $u_0=w_{n-1}$ and  $u_{i} = s_{n-i} u_{i-1}s_{n-i} $ for
$1\leq i \leq \lfloor \frac{n-1}{2}\rfloor $, and if $n=2k$ is even define
$u_{k} = s_ku_{k-1}  = u_{k-1}s_k$, so
that  $u_{\lfloor n/2\rfloor} = w_n$.
By considering the sequence of matchings on $[n]$ representing  $u_0,u_1,u_2,\dots$ and consulting
Corollary \ref{taubruhat-cor} and Table \ref{ct-fig}, one checks   for
 $1\leq i \leq \lfloor \frac{n}{2}\rfloor $ that
$(i, n-i) $ is a cycle of $u_{i-1}$, that $\hat \Phi^+(u_{i-1}, n-i) = \{u_{i}\}$, and that
$\hat \Phi^-(u_{i-1},i) \cap \I_\infty = \varnothing$.
It follows by
Theorem \ref{invmonk-thm}
that
$x_{(i,n-i)} \iS_{u_{i-1}}  = \iS_{u_{i}}$ for $1\leq i \leq \lfloor \frac{n}{2}\rfloor $,
and so the desired formula
follows by induction since  $ \iS_{u_0} =\iS_{w_{n-1}}$.
\end{proof}

\subsection{Extending the Little map}\label{little-sect}

Given a map $w : \ZZ\to \ZZ$ and $N \in \ZZ$, write $w \gg N$ for the map $\ZZ \to \ZZ$ with
$i\mapsto w(i-N) + N$.
It follows from \cite[Theorem 1.1]{BJS} that 
if $z \in \I_\ZZ$ and $k= \ellhat(z)$ then $ |\iR(z)|$ is the coefficient of $x_1x_2\cdots x_k$
in $\iS_{z \gg N}$ for all sufficiently large $N \in \NN$; see the discussion following \cite[Eq.\ (1.3)]{HMP1}.
It is clear by definition that $\tau_{ij}(y) \gg N = \tau_{i+N,j+N}(y \gg N)$ for all
$N \in \ZZ$, and it follows that 
the map $z \mapsto z\gg N$ is a bijection $\hat\Phi^+(y,q) \to \hat\Phi^+(y\gg N, q+N)$
and 
$\hat\Phi^-(y,p) \to \hat\Phi^-(y\gg N, p+N)$.
Using these facts, it is a simple exercise to derive the following 
identity from Theorem \ref{invmonk-thm}:

\begin{proposition}\label{little3-cor}
If $y \in \I_\ZZ$ and $(p,q) \in \Cyc_\ZZ(y)$, then
\[\sum_{z \in \hat\Phi^-(y,p)} |\iR(z)|  = \sum_{ z \in \hat\Phi^+(y,q)} |\iR(z)|.\]
\end{proposition}

%\begin{proof}
%Fix $y \in \I_\ZZ$ and $(p,q) \in \Cyc_\ZZ(y)$.
%The map $z \mapsto z\gg N$ is a bijection $\hat\Phi^+(y,q) \to \hat\Phi^+(y\gg N, q+N)$
%and
%$\hat\Phi^-(y,p) \to \hat\Phi^-(y\gg N, p+N)$, so by Theorem \ref{invmonk-thm} it holds that 
%\be\label{little-eq}
% \sum_{ z \in \hat\Phi^+(y,q)} \iS_{z\gg N}  = x_{(p+N,q+N)} \iS_{y\gg N} +\sum_{z \in \hat\Phi^-(y,p)} \iS_{z\gg N}
% \ee
%for all $N \in \ZZ$ such that $y \gg N \in \I_\infty$.
%It is a consequence of \cite[Theorem 1.1]{BJS} that 
%if $z \in \I_\ZZ$ and $k= \ellhat(z)$ then $ |\iR(z)|$ is the coefficient of $x_1x_2\cdots x_k$
%in $\iS_{z \gg N}$ for all sufficiently large $N \in \NN$; see the discussion following \cite[Eq.\ (1.3)]{HMP1}.
%The desired identity therefore follows 
%by comparing the coefficients of $x_1x_2\cdots x_k$ on both sides of \eqref{little-eq} when $k = \ellhat(y)+1$
%and $N$ is large enough.
%\end{proof}

A bijective proof for the ``reduced words'' version of this identity
 is known via the \emph{Little map} introduced in \cite[\S5]{Little}.
We show in this section how that bijection may be extended to involution words
to prove the preceding result.
Our arguments and notation are parallel to that of Lam and Shimozono \cite{LamShim}.

Let $\a = (s_{a_1},s_{a_2},\dots,s_{a_k})$ be a sequence of simple transpositions and $i \in [k]$.
Write $\del_i(\a)$ for the subsequence $(s_{a_1}, \dots, \wh{s_{a_i}}, \dots, s_{a_k})$
obtained from $\a$ by deleting the $i$th entry.
The pair $(\a, i)$ is a \emph{$y$-marked involution word} for some  $y \in \I_\ZZ$
if $\del_i(\a)\in \iR(y)$. If $\a \in \hat \cR(z)$ for some $z \in \I_\ZZ$
then we say that $(\a,i)$ is \emph{reduced}.

\begin{lemma} \label{invlem:nearly-reduced}
Let $y \in \I_\ZZ$ and suppose $(\a,i)$ is $y$-marked involution word which is not reduced.
There is then a unique index $j \neq i$ such that $(\a,j)$ is a $y$-marked involution word.
\end{lemma}

\begin{proof}
Write $\a = (s_{a_1},s_{a_2},\dots,s_{a_k})$ and
let $u = s_{a_1} \cdots \widehat{s_{a_i}}\cdots s_{a_k} \in \cA(y)$.
If $\a \notin \cR(v)$ for all $v \in S_\infty$,
then \cite[Lemma 21]{LamShim} asserts that there is a unique index $j\neq i$
such that $\del_j(\a)\in \cR(u) \subset \iR(y)$.
If $\a \in \cR(v)$ for some  $v \in S_\ZZ$, then $u\lessdot ut_1=v$ for
$t_1= s_{a_k}\cdots s_{a_{j+1}} s_{a_j} s_{a_{j+1}} \cdots s_{a_k}$.
Since $v \notin \cA(z)$ for all $z \in \I_\ZZ$ by hypothesis,
Theorem \ref{mir2-thm} implies that there is a unique transposition $t_2 \in S_\ZZ$
such that $u\neq ut_1t_2 \in \cA(y)$.
By the Strong Exchange Condition  \cite[Theorem 5.8]{Hu}, there is a unique  $j \in [k]$
with $(s_{a_1}, \dots, \widehat{s_{a_j}},\dots, s_{a_k}) \in \cR(ut_1t_2) \subset \iR(y)$,
and $j\neq i$ since $u\neq ut_1t_2$.
\end{proof}

%\begin{remark}
%It would interesting to find an efficient way of identifying the index $j \neq i$ characterized
%by the lemma directly from the involution wiring diagram of $\del_i(\a) \in \iR(y)$, 
%as defined in Section~\ref{invword-sect}. We presently do not know of any simple means of doing this.
%\end{remark}

Continue to let $(\a,i)$ denote a $y$-marked involution word.
If $\a \in \cR(w)$ for some $w \in S_\ZZ$ but $\a \notin \iR(z)$ for all $z \in \I_\ZZ$
(so that $(\a,i)$ is not reduced), then we say that $(\a,i)$ is \emph{nearly reduced}.

\begin{lemma}\label{nr-lem2}
Suppose $(\a,i)$ is a nearly reduced $y$-marked involution word.
Write $j\neq i$ for the index such that $(\a,j)$ is  a $y$-marked involution word
and let
 $u,v \in \cA(y)$ and $w \in S_\ZZ$ be such that
 \[
 \del_i(\a) \in \cR(u) \qquand \del_j(\a) \in \cR(v)
 \qquand \a \in \cR(w).
 \]
If  $(p,q) \in \Cyc_\ZZ(y)$ is such that $w \in \Phi^\pm(u,p) \cup \Phi^\pm(u,q)$,
then $w \in \Phi^\mp(v,p) \cup \Phi^\mp(v,q)$.
\end{lemma}

\begin{proof}
Since $u\lessdot w$ and $v\lessdot w$ and $w \notin \cA(z)$ for all $z \in \I_\ZZ$,
the result follows by Theorem \ref{mir2-thm}.
\end{proof}

Fix a $y$-marked involution word $(\a,i)$, and let $\b$ denote the sequence formed by
decrementing the index of the $i$th entry of $\a$, so that if $\a=(s_{a_1},\dots,s_{a_k})$
then $\b = (s_{a_1},\dots,s_{a_{i-1}}, s_{a_i-1}, s_{a_{i+1}},\dots,s_{a_k})$.
With respect to this notation, we define
$(\a,i) \push = (\b,j)$
where  $j$ is given  as follows:
\begin{itemize}
\item If $\b$ is an involution word then $j=i$.
\item Otherwise, $j$ is the  index distinct from $i$ such that
$(\b,j)$ is a $y$-marked involution word.
\end{itemize}
The index $j$ is well-defined and uniquely determined by Lemma \ref{invlem:nearly-reduced}.
The operation $\push$ is invertible and we denote its inverse by $\pull$.
%, so that $(\b,j)\pull = (\a,i)$ when $(\a,i)\push = (\b,j)$.
Explicitly, if $\b = (s_{b_1},\dots,s_{b_k})$ and $j\in[k]$
are such that $(\b,j)$ is a $y$-marked involution word,
then we have $(\b,j)\pull = (\a,i)$ where
 $i=j$ if $\b$ is an involution word
and where otherwise $i$ is the unique index distinct from $j$ such that $(\b,i)$ is a $y$-marked involution word,
and in both cases $\a = (s_{b_1},\dots,s_{b_{i-1}}, s_{b_i+1}, s_{b_{i+1}},\dots,s_{b_k})$.

\begin{example}\label{little-ex}
We write $a_1\cdots \boxed{a_i}\cdots a_k$ in place of the more cumbersome notation
$(\a,i)=( (s_{a_1},\dots,s_{a_k}),i)$.
Then $21\boxed{3}3$ is a $y$-marked involution word for $y=(1,4)$,
and we have
\[21\boxed{3}3 \push  = \boxed{2}123 \qquand 21\boxed{3}3\pull = 213\boxed{4} \]
which are respectively  nearly reduced and  reduced.
\end{example}

Given a marked involution word $(\a,i)$ and $N \in \NN$,
we write $(\a,i)\push^N = (\a,i) \push \push \cdots \push $
for the effect of applying $\push$ to $(\a,i)$ exactly $N$ times.
Define $(\a,i)\pull^N$ similarly.

\begin{lemma}\label{little-prop}
Suppose $(\a,i)$ is a $y$-marked involution word.
There are then positive integers $M,N$ such that the $y$-marked
involution words $(\a,i)\push^M $ and $(\a,i)\pull^N$
are  reduced.
\end{lemma}

\begin{proof}
%Let $(\a,i)$ be a $y$-marked involution word.
If $(\b,j) = (\a,i)\push$ is not reduced, then $i\neq j$ so $s_{a_i-1}$ appears in
$\del_j(\b) \in \iR(y)$.
Hence if $(\a,i)\push^M $ is not reduced for all $M>0$, then $s_k$ appears in an involution word for $y$
for every sufficiently small $k$,
which is clearly impossible since 
$|\iR(y)| \leq |\cR(y)| < \infty$.
We reach a similar contradiction if we assume that $(\a,i)\pull^N$ is not reduced for all $N>0$.
\end{proof}

\begin{definition}
The \emph{involution Little bump} $\iB$ is the operation on marked involution words defined by
$\iB(\a,i) = (\a,i)\push^N$
where $N> 0$ is the least positive integer such that $(\a,i)\push^N$ is reduced.
\end{definition}

Note that since $\push$ is an invertible operation, the involution Little bump $\iB$ is also
invertible.

\begin{example}
With our notation as in Example \ref{little-ex}, the sequence $324\boxed{5}$ is a reduced
$y$-marked involution word for $y=(2,5)$.
Applying $\push$ successively gives the sequence of marked words
\[
324\boxed{5}
\quad\to\quad
32\boxed{4}{4}
\quad\to\quad
\boxed{3}23{4}
\quad\to\quad
2\boxed{2}3{4}
\quad\to\quad
2\boxed{1}3{4}
\]
the last of which is reduced. Thus $\iB(324\boxed{5}) = 324\boxed{5} \push^4 =2\boxed{1}3{4}$.
Note that $(s_3,s_2,s_4,s_5) \in \iR(z)$ for $z = (2,6) = \tau_{5,6}(y)$
while $(s_2,s_1,s_3,s_4) \in \iR(z)$ for $z = (1,5) = \tau_{1,2}(y)$.
\end{example}

\begin{remark}
This example illustrates a stronger property than   Lemma~\ref{little-prop} which appears to hold in general:
namely, if $(\a,i)$ is a marked involution word of length $k$ then 
it seems that we can always find positive integers  $M,N \leq k$ such that
$(\a,i)\push^M $ and $(\a,i)\pull^N$
are reduced.
More specifically, during the sequence of $\push$-operations that compose $\iB$, each entry
of a marked involution word appears to be decremented at most once.
\end{remark}

Fix $y,z \in \I_\ZZ$. If
$(\a,i)$ is a reduced $y$-marked involution word such that $\a \in \iR(z)$, then
$y \lessdot_\I z$ by Theorem \ref{invbruhat-thm}.
By the same result,
conversely, if $y\lessdot_\I z$ then for any involution word
$\a\in \iR(z)$, there exists a unique index $i$ such that $(\a,i)$ is a
reduced
$y$-marked word;  in this situation we define $\iB_y(\a) = \b$ where $\iB(\a,i) = (\b,j)$.
Since $\iB$ is evidently invertible on $y$-marked involution words,
the operation $\iB_y$ defines a bijection
$ \iB_y : \bigcup_{ z } \iR(z) \to \bigcup_{z} \iR(z)$, with both unions  over $z \in \I_\ZZ$
with $y\lessdot_\I z$.
We refer to $\iB_y$ as the \emph{involution Little map}, in reference to the bijection defined
in \cite{Little}. This map affords a bijective proof of Proposition~\ref{little3-cor} in view of
the following.

\begin{theorem}\label{little-thm}
Let $y \in \I_\ZZ$ and $(p,q) \in \Cyc_\ZZ(y)$.
The map  $\iB_y$   restricts to a bijection
\[ \bigcup_{z \in \hat\Phi^+(y,q)} \iR(z) \to \bigcup_{z \in \hat\Phi^-(y,p)} \iR(z).\]
\end{theorem}

\begin{proof}
Let $(\a,i)$ be a $y$-marked involution word and let $u \in \cA(y)$ be such that $\del_i(\a) \in \cR(u)$.
Define
 $\cB(\a,i) = (\a,i)\push^M$ where $M>0$ is the least positive integer such that
$ (\a,i)\push^M$ is reduced or nearly reduced,
and let
$
X^\pm(u;p,q) = \bigcup_{w \in \Phi^\pm(u,p)\cup  \Phi^\pm(u,q)} \cR(w).
$
The operation $\cB$ corresponds to the ordinary \emph{Little bump} as described in \cite{LamShim}, and
if we let $(\b,j) = \cB(\a,i) $
then it follows from
\cite[Lemma 7]{Little}
that
the map $\a \mapsto \b$ restrict to a bijection
\be\label{little-eq}
X^+(u;p,q)
\to
X^-(u;p,q).
\ee
Note that for any $\a \in X^+(u;p,q)$ there exists a unique index $i$ with $(\a,i)$ is a
$y$-marked involution word, so the map \eqref{little-eq} induced by $\cB$ is well-defined.
Also observe that if $(\b,j) = \cB(\a,i) $ is nearly reduced, then $\iB(\a,i) = \iB(\b,j)$.

Fix $\a \in \bigcup_{z \in \hat\Phi^+(y,q)} \iR(z)$
and write $i$ for the unique index such that $\del_i(\a) \in \iR(y)$.
Define  $(\a_0,i_0) = (\a,i)$ and
for $t>0$ let $(\a_t,i_t) = \cB(\a_{t-1},i_{t-1})$.
Write $k$ for the first positive integer such that $(\a_k,i_k)$ is reduced, so that $\iB_y(\a) = \a_k$,
and for each $0\leq t \leq k$ let $u_t \in \cA(y)$ 
denote the permutation with $\del_{i_t}(\a_t) \in \cR(u_t)$.
It follows by
Lemma \ref{inspection-lem}(b) and Theorem \ref{tau-thm}(a) 
 that $ \a_0=\a \in X^+(u_0;p,q)$.
By induction, Lemma \ref{nr-lem2}, and the observations in the previous paragraph,
it follows in turn that
\[ \a_{t} \in X^-(u_{t-1};p,q) \cap X^+(u_{t};p,q)
\qquad\text{for }0<t<k.
\]
In particular, $\a_{k-1}  \in X^+(u_{k-1};p,q)$,
so by \eqref{little-eq} we have
$\a_k \in X^-(u_k;p,q)$.
Since $(\a_k,i_k)$ is reduced,
Theorem \ref{tau-thm}(a)
and Lemma \ref{phi+lem}
imply that 
$\iB_y(\a)=\a_k \in \iR(z)$ for some $z \in  \hat\Phi^-(y,p)$.
We conclude that $\iB_y$ restricts to an injective map
$ \bigcup_{z \in \hat\Phi^+(y,q)} \iR(z) \to \bigcup_{z \in \hat\Phi^-(y,p)} \iR(z)$.
As our arguments  apply equally well to the inverses of involution Little bumps, this
map is a bijection.
\end{proof}

\begin{example}
Some examples are helpful for unpacking the preceding theorem and its proof.
\begin{enumerate}
\item[(i)] 
We have $(s_5,s_3,s_4,s_2,s_3) \in \cR(146235) \subset \iR(156423)$.
It follows from Example~\ref{ex:transition}
that this reduced word belongs to the domain of the map $\iB_{15432}$.
We compute its image by
\[
\boxed{5}3423 \to 43\boxed{4}23 \to 4\boxed{3}323 \to 4232\boxed{3} \to 423\boxed{2}2 \to 423\boxed{1}2
\]
and get $\iB_{15432}(s_5,s_3,s_4,s_2,s_3) = (s_4,s_2,s_3,s_1,s_2) \in \cR(35124) \subset \iR(45312)$.

To better understand this computation, we explain some of the intermediate steps.
Although $(s_4,s_3,s_4,s_2,s_3)$ is not an involution word for any $z \in \I_\infty$,
this sequence is a reduced word for the permutation $15423$.
Removing the marked entry from $\boxed{5}3423$ gives an element of $\cR(14523)$ 
and we have $(s_4,s_3,s_4,s_2,s_3) = \cB_{14523}(s_5,s_3,s_4,s_2,s_3)$.
In the steps from $43\boxed{4}23$ to $423\boxed{1}2$, the intermediary words are not reduced,
and we have $(s_4,s_3,s_2,s_3) \in \cR(13542)$ and $\cB_{13542}(s_4,s_3,s_4,s_2,s_3) = (s_4,s_2,s_3,s_1,s_2)$.

Returning to the direct computation in Example~\ref{ex:transition}, we have
\[
x_3\fkS_{14523} = \fkS_{146235} - \fkS_{15423} \qquand x_3 \fkS_{13542} = \fkS_{15423} - \fkS_{35124}.
\]
These equalities can be explained at the level of reduced words by showing that
$\cB_{14523}$ defines a bijection $\cR(146253) \to \cR(15423)$, while $\cB_{13542}$
defines a bijection $\cR(15423) \to \cR(35124)$; see~\cite{Little} for details.
The fact that $\iB_{15432}(s_5,s_3,s_4,s_2,s_3)=\cB_{13542}(\cB_{14523}(s_5,s_3,s_4,s_2,s_3))$ realizes
the cancellations which occur in Example~\ref{ex:transition}.

\item[(ii)] 
Using Theorem~\ref{invmonk-thm}, we compute that
\[
x_4 \iS_{3214765} = x_{(4,4)} \iS_{3214765} = \iS_{3217564} + \iS_{3216745} - \iS_{4231765} - \iS_{3412765}.
\]
Both $(s_4,s_5,s_6,s_1,s_2)$ and $(s_4,s_5,s_6,s_2,s_1)$ belong to $\iR(3217564)$
and we compute that 
\[\iB_{3214765}(s_4,s_5,s_6,s_1,s_2) = (s_3,s_5,s_6,s_1,s_2) \in \iR(3412765)\]
while 
\[\iB_{3214765}(s_4,s_5,s_6,s_2,s_1) = (s_3,s_5,s_6,s_2,s_1) \in \iR(4231765).\]
For every pair $(y,z) \in \{3217564,3216745\} \times \{4231765,3412765\}$,
one can find similar examples to show that $\iB_{3214765}(\iR(y)) \cap \iR(z) \neq \varnothing$.
\end{enumerate}
\end{example}

If $(W,S)$ is a Coxeter system and $*$ is an involution of $W$ preserving $S$,
then we refer to the triple $(W,S,*)$ as 
a \emph{twisted Coxeter system} 
and define $\I_*(W) = \{ w \in W: w^{-1} = w^*\}$.
The notion of an involution word
 extends without difficulty to elements of $\I_*(W)$; see \cite{HMP2} for the precise definition.
Lemma \ref{invlem:nearly-reduced}, which only applies in type $A$,
 is formally very similar to Lam and Shimozono's \cite[Lemma 21]{LamShim},
 which concerns reduced words in arbitrary Coxeter groups.
We suspect that this  more general version of our lemma holds:

\begin{conjecture} Let $(W,S,*)$ be an arbitrary twisted Coxeter system.
Suppose the sequence $(s_{1}, \dots, \widehat{s_{k}},\dots, s_{m})$ is an involution word
for some $y \in \I_*(W)$, but $(s_{1}, \dots, s_{m})$ is  not an involution word for any $z \in \I_*(W)$.
Then there is a unique index $j \neq k$ such that
$(s_{1}, \dots, \widehat{s_{j}},\dots, s_{m})$ is an involution word.
Moreover, it holds that
 $(s_{1}, \dots, \widehat{s_{j}},\dots, s_{m})$ is an involution word for $y$.
\end{conjecture}

\section{Formulas in the symplectic case}\label{fpfinvtrans-sect}

Recall that $\F_n$ for $n \in \PP$ denotes the set of  elements $z \in \I_n$ with $z(i) \neq i$
for all $i \in [n]$. Note that $\F_n$ is empty if $n$
is odd. With slight abuse of notation, we define $\F_\infty$ and $\F_\ZZ$ as the $S_\infty$- and
$S_\ZZ$-conjugacy classes of the permutation $\wfpf:\ZZ \to \ZZ$ given by
\[\wfpf : i \mapsto i - (-1)^i.\]
If $z \in \F_\ZZ$ and $N \in \ZZ$, then we define $z\gg N$ as the permutation of $\ZZ$ with $i \mapsto z(i-N)+N$,
exactly as in Section~\ref{little-sect}; observe that  
we then have $z \gg N \in \F_\ZZ$
 if and only if $N$ is even.
While technically $\F_n\not\subset \F_\infty$ according to our definition, there is a natural
inclusion
\be
\label{ifpf-def}
\ifpf: \F_n \hookrightarrow \F_\infty
\ee
 mapping $z \in \F_n$ to the permutation of $\ZZ$ whose respective restrictions to $[n]$ and to
 $\ZZ\setminus [n]$ coincide with those of $z$ and $\wfpf$.
In symbols, $\ifpf(z) = z \cdot \wfpf \cdot s_1 \cdot s_3 \cdot s_5\cdots s_{n-1} $.
It holds that $\F_\infty = \bigcup_{n \in 2\PP} \ifpf(\F_n)$,
and we often identify $z \in \F_n$ with its image $\ifpf(z) \in \F_\infty$ without comment.

Let $\cAfpf(z)$ for $z \in \F_\ZZ$ be the set of permutations $w \in S_\ZZ$ of minimal length
such that $z=w^{-1} \wfpf w$,
and
define $\cRfpf(z)  = \bigcup_{w \in \cAfpf(z)} \cR(w)$.
We sometimes refer to elements of $\cRfpf(z)$ as \emph{FPF-involution words}. The sets $\cRfpf(z)$ and $\cAfpf(z)$ have been previously studied in \cite{CJW,HMP2,RainsVazirani}.

\begin{example}
For $z \in \F_n$, we set $\cAfpf(z) = \cAfpf(\ifpf(z))$ and $\cRfpf(z) = \cRfpf(\ifpf(z))$.
Then, for example, we have $\cRfpf(4321) = \{ (s_2,s_1),(s_2,s_3)\}$ and
$\cAfpf(4321) = \{231,1342\}$.
\end{example}

The main result of this section is a transition formula for the following polynomials:

\begin{definition} \label{fS-def}
The \emph{FPF involution Schubert polynomial} of $z \in \F_\infty$ is
$ \Sfpf_z = \sum_{w \in \cAfpf(z)} \fkS_w.$ 
\end{definition}

\begin{example}
Set $\Sfpf_z = \Sfpf_{\ifpf(z)}$ for $z \in \F_n$.
Then $\Sfpf_{4321} =  \fkS_{312} +  \fkS_{1342} =x_1^2 + x_1 x_2 + x_1 x_3 + x_2 x_3$.
\end{example}

Define the sets $\inv(z)$, $\DesR(z)$, $\Cyc_\ZZ(z)$, and $\Cyc_\PP(z)$ for $z \in \F_\ZZ$ exactly
as for elements of $S_\ZZ$.
The FPF involution Schubert polynomials are evidently homogeneous.
By \cite[Corollary 3.13]{HMP1}, this homogeneous family has the following characterization
via divided differences.

\begin{theorem}[See \cite{HMP1}]
\label{fpfthm}
The FPF involution Schubert polynomials $\{ \Sfpf_z\}_{z \in \F_\infty}$
are the unique family of homogeneous polynomials indexed by $\F_\infty$
such that if $i \in \PP$ and $s=s_i$ then
\be\label{fpf-eq}
\Sfpf_{\wfpf} = 1
\qquand
\partial_i \Sfpf_z = \begin{cases}
\Sfpf_{szs} &\text{if $s \in \DesR(z)$ and $(i,i+1)\notin\Cyc_\ZZ(z)$} \\
0&\text{otherwise}.\end{cases}
\ee
\end{theorem}

Wyser and Yong defined these polynomials in \cite{WY},
where they were denoted $\Upsilon_{z; (\GL_{n}, \Sp_n)}$.
When $n$ is even, Wyser and Yong showed that the FPF involution Schubert polynomials
 may be identified with cohomology representatives of the
$\Sp_n(\CC)$-orbit closures in $\Fl(n)$, and so coincide with an older construction of Brion \cite{Brion98}.
 As with the
involution Schubert polynomials in Section \ref{invtrans-sect},
it is a nontrivial result to show that the polynomials given by
Definition \ref{fS-def} coincide with Wyser and Yong's polynomials in \cite{WY}.
This follows from \cite[Section 3.4]{HMP1},
or alternatively from Corollary \ref{wy-fpf-cor}.

\subsection{Bruhat order on FPF involutions}

We present a transition formula  similar to Theorem \ref{thm:involution-transition-formula} for the polynomials
$\Sfpf_z$ in the next section.
In preparation for this, we record some properties of the Bruhat order on $\F_\ZZ$.
Let
$\inv_\fpf(z) = \inv(z) - \Cyc_\ZZ(z)$ for $z \in \F_\ZZ$.
It follows as an exercise  that
$\inv_\fpf(z)$ is a finite set with an even number of elements,
which is empty if and only if $z=\wfpf$.
For $z \in \F_\ZZ$, we may therefore define
\[\ellfpf(z) = \tfrac{1}{2}|\inv_\fpf(z)|
\qquand
\DesF(z) = \{ s_i  \in \DesR(z) : (i,i+1)\notin \Cyc_\ZZ(z)\}.\]
These notations are related by the following lemma, whose elementary proof is left to the reader.

\begin{lemma}\label{fpfdes-lem}
If $z \in \F_\ZZ$
then
$
\ellfpf(szs)
=
\begin{cases}
\ellfpf(z)-1 &\text{if }s \in \DesF(z) \\
\ellfpf(z) &\text{if }s \in \DesR(z) -\DesF(z) \\
\ellfpf(z)+1&\text{if }s \in \{s_i : i \in \ZZ\} - \DesR(z).
\end{cases}
$
\end{lemma}

We deduce by induction  that $\ellfpf(z)$ is the common length of the elements $\cAfpf(z)$
and $\cRfpf(z)$, and therefore also the degree of
$\Sfpf_z$.  Define the \emph{Bruhat order} $<$ on $\F_\ZZ$ as the weakest partial order
with $z < tzt$ if $z \in \F_\ZZ$ and $t \in S_\ZZ$ is a transposition such that
$\ellfpf(z) < \ellfpf(tzt)$.
Both $\ifpf(\F_n)$ and $\F_\infty$ are evidently lower ideals in $(\F_\ZZ,<)$.
Some other properties of $<$ include the following:

\begin{theorem}[Rains and Vazirani \cite{RainsVazirani}]
\label{fpfbruhat-thm}
Let $n \in 2\PP$.
The following properties of $(\F_\ZZ,<)$ hold:
\ben
\item[(a)] $(\F_\ZZ,<)$ is a graded poset with rank function $\ellfpf$.
\item[(b)] If $y,z \in \F_n\subset \I_\infty$ then $y\leq z$  holds in $(S_\ZZ,<)$
 if and only if $\ifpf(y) \leq \ifpf(z)$  holds in $(\F_\ZZ,<)$.

\item[(c)] Fix $y,z \in \F_\ZZ$, $(s_{i_1},\dots,s_{i_k}) \in \cRfpf(z)$, and $w \in \cAfpf(z)$.
The following are then equivalent:
\ben
\item[1.] $y\leq z$.
\item[2.] A subword of  $ (s_{i_1},\dots,s_{i_k}) $ belongs to $\cRfpf(y)$.
\item[3.] An element $v \in \cAfpf(y)$ exists such that $v \leq w$.
\een

 \een
\end{theorem}

\begin{proof}
By \cite[Theorem 4.6]{RainsVazirani}, $\F_\ZZ$ with the height function $\ellfpf$
is an example of what Rains and Vazirani call a \emph{quasiparabolic $W$-set} (with $W=S_\ZZ$).
The Bruhat order we have defined on $\F_\ZZ$ is the same as the ``Bruhat order''
which is defined in \cite[\S5]{RainsVazirani} for any quasiparabolic set.
Part (a) is a special case of the general result about this order
\cite[Proposition 5.16]{RainsVazirani}. 
Part (b) is equivalent to the assertion that the Bruhat order of $S_n$ restricted
to $\F_n$ coincides with the weakest partial order $\prec$ on $\F_n$ with $w\prec twt$
for any transposition $t \in S_n$ with  $\ell(w) <\ell(twt)$.
This assertion follows from \cite[Proposition 5.17 and Remark 5.18]{RainsVazirani}.
Part (c) is equivalent to \cite[Theorem 5.15]{RainsVazirani}.
\end{proof}

We write $y \lessdot_\F z$ for $y,z \in \F_\ZZ$ if  $\{y\}= \{ w \in \F_\ZZ : y\leq w < z\}$.
Similarly, if $y,z \in \F_n$ for some $n \in 2\PP$, then we write $y \lessdot_\F z$ if
$\ifpf(y) \lessdot_\F \ifpf(z)$; by Theorem \ref{fpfbruhat-thm}(b),
this holds if and only if $z$ covers $y$ in the order given by restricting the usual
Bruhat order on $S_\ZZ$ to $\F_n$.
We will need an explicit description of these covering relations. However,
since by definition $y \lessdot_\F z$ only if $z = tyt$ for a transposition $t \in S_\ZZ$,
the situation here is less complicated than in Section \ref{invtrans-sect}.

\begin{example}\label{f4-ex}
 The set $\F_4 = \{ (1,2)(3,4) < (1,3)(2,4)<(1,4)(2,3) \}$  is totally ordered by $<$.
\end{example}

For a fixed-point-free involution $z \in \F_\ZZ$, we say that distinct cycles
$(a,b),(i,j)\in \Cyc_\ZZ(z)$ with $a<i$
are \emph{crossing} if $a<i<b<j$ and \emph{nesting} if $a<i<j<b$.
The following basic properties are equivalent to \cite[Lemma 2.2 and Corollary 2.3]{Bertiger}.

\begin{lemma}[See \cite{Bertiger}]
\label{crossnest-lem}
If $z \in \F_\ZZ$ then $\ellfpf(z) = 2n+c$ where $n$ is the number of unordered
pairs of nesting cycles of $z$, and $c$ is the number of unordered pairs of crossing cycles of $z$.
\end{lemma}

\begin{proposition}[See \cite{Bertiger}]
\label{fpfbruhatcover-thm}
Let $y \in \F_\ZZ$. Fix integers $i<j$ and define $A = \{i,j,y(i),y(j)\}$ and $z=(i,j)y(i,j)$.
Then $\ellfpf(z) = \ellfpf(y)+1$ if and only if the following conditions hold:
\ben
\item[(a)]  $y(i)<y(j)$ and no $e \in \ZZ$ exists with  $i<e<j$ and $y(i) < y(e)< y(j)$.
\item[(b)] Either $[y]_A= (1,2)(3,4) \lessdot_\F [z]_A= (1,3)(2,4)$ or 
$[y]_A=(1,3)(2,4) \lessdot_\F [z]_A= (1,4)(2,3)$.

\een

\end{proposition}

\begin{remark}
If condition (a) holds then $(i,j) \notin \Cyc_\ZZ(y)$ so  necessarily  $|A|=4$, and condition (b)
asserts that $[y]_A \lessdot_\F [z]_A$, which occurs if and only if these involutions coincide with
\[ \arcstart
{
*{.}  \arc{.6}{r}   & *{.}   & *{.} \arc{.6}{r}  & *{.}
}
\arcstop
\ \lessdot_\F \
\arcstart
{
*{.}  \arc{.8}{rr}   & *{.} \arc{.8}{rr}   & *{.} & *{.}
}
\arcstop
\qquord
\arcstart
{
*{.}  \arc{.8}{rr}   & *{.} \arc{.8}{rr}   & *{.} & *{.}
}
\arcstop
\ \lessdot_\F \
\arcstart
{
*{.}  \arc{1.0}{rrr}   & *{.} \arc{.4}{r}   & *{.} & *{.}
}
\arcstop.\]
In the first case we must have $[(i,j)]_A \in \{(1,4),(2,3)\}$, and in the second
$[(i,j)]_A \in \{(1,2),(3,4)\}$.
\end{remark}

\begin{corollary}\label{fpf<-cor}
If $y \in \F_\ZZ$ and $i<j$ are integers such that $y(i)<y(j)$, then $y < (i,j)y(i,j)$.
\end{corollary}

\begin{proof}
It suffices by Theorem~\ref{fpfbruhat-thm}(b) to show that if $y \in \F_n$ for some $n \in 2\PP$ and
$1\leq i < j \leq n$ are such that $y(i)<y(j)$, then $y < (i,j)y < (i,j)y(i,j)$.
This is easy to check after noting that if $i<j$ and $y(i)<y(j)$ for $y \in \F_n$,
then $\{i,j\} \cap \{y(i),y(j)\} = \varnothing$ since
 $(i,j)$ is not a cycle of $y$.
\end{proof}

\subsection{Transition formulas}\label{fpf-trans-sect}

We begin this section with two elementary lemmas.

\begin{lemma}\label{fpfbasic-lem}
Let $y \in \F_\ZZ$ and let $r,t \in S_\ZZ$ be transpositions with $y\neq tyt$.
Then $ryr=tyt$ if and only if $r\in \{t, yty\}$. \end{lemma}

\begin{proof}
Checking this assertion is a simple exercise which is left to the reader.
\end{proof}

\begin{lemma}\label{fpfmonk-lem1}
Let $y,z \in \F_\ZZ$ and let $s \in S_\ZZ$ be a simple transposition.
If $y\lessdot_\F z$ and $s \notin \DesR(y)$, then either $s \notin \DesR(z)$ or $z=sys$.
\end{lemma}

\begin{proof}
We may assume without loss of generality that $y,z \in \F_\infty$ and $s \in S_\infty$.
By Lemma \ref{fpfdes-lem}, we have $s \in \DesR(z)$ if and only if $\ellfpf(szs) \leq \ellfpf(z)$.
Assume $s \notin\DesR(y)$ and 
let $r \in S_\infty$ be a transposition such that $\ellfpf(ryr) =\ellfpf(sys)= \ellfpf(y)+1$
and $\ellfpf(sryrs) \leq \ellfpf(sys)$; it suffices to show that $ryr=sys$.
If $\ellfpf(sryrs) < \ellfpf(sys)$ then this assertion follows from the fact that
$(\F_\infty, \ellfpf)$ is a quasiparabolic $S_\infty$-set (cf.\ \cite{RainsVazirani}),
which  holds by  \cite[Theorem 4.6]{RainsVazirani}.
Suppose alternatively that $\ellfpf(sryrs) = \ellfpf(ryr)$.
Lemma \ref{fpfdes-lem} then implies that $sryrs=ryr$, so $s$ must be a cycle of $ryr$
but not of $y$. By Proposition~\ref{fpfbruhatcover-thm}, the only way this can occur is if
$y$ has cycles of the form $(a,i+1)$ and $(i,b)$ with $a<i<i+1<b$, and it holds that
$ryr = (a,i)y(a,i)$ and $s=(i,i+1)$. But this would imply that $s \in \DesR(y)$,
which is a contradiction.
\end{proof}

Given $y \in \F_\ZZ$ and $r \in \ZZ$, we define
\[
\ba
\Pfpf^+(y,r) &=
\left\{
z\in \F_\ZZ : y\lessdot_\F z \text{ and } z = (r,j)y(r,j)\text{ for an integer }j>r
\right\}
\\
\Pfpf^-(y,r) &=
\left\{
z \in \F_\ZZ : y \lessdot_\F z \text{ and } z =(i,r)y(i,r)\text{ for an integer }i<r
\right\}.
\ea
\]
Note by Theorem~\ref{fpfbruhat-thm} that the condition `` $y\lessdot_\F z$'' may be replaced
by ``$\ellfpf(z) = \ellfpf(y)+1$'' without changing the meaning of these sets.
As usual, we let $\Pfpf^\pm(y,r) = \Pfpf^\pm(\ifpf(y),r)$ for $y \in \F_n$.

\begin{example} \label{ex:fpf-phi}
If $y = (1,5)(2,4)(3,6)(7,8) \in \F_8$ then, identifying $\F_8$ with
$\ifpf(\F_8)\subset \F_\infty$, we have
\begin{gather*}
 \Pfpf^+(y,6) 
= \{ (6,7)y(6,7)\} = \{(1,5)(2,4)(3,7)(6,8)\}
\\
\Pfpf^-(y,3) = \{ (2,3)y(2,3),\, (1,3)y(1,3) \} = \{ (1,5)(2,6)(3,4)(7,8),\, (1,6)(2,4)(3,5)(7,8) \}.
\end{gather*}
\end{example}

The sets $\Pfpf^\pm(y,r)$ have some properties in common with $\hat\Phi^\pm(y,r)$:

\begin{proposition}
If $y \in \F_\ZZ$ and $(p,q) \in \Cyc_\ZZ(y)$ then 
$\Pfpf^-(y,p)$ and $\Pfpf^+(y,q)$ are both nonempty.
\end{proposition}

\begin{proof}
The proof is similar to that of Proposition~\ref{nonempty-prop}; we skip the details.
\end{proof}

\begin{lemma}\label{psi+lem}
If $y \in \F_\ZZ$ and $(p,q) \in \Cyc_\ZZ(y)$
then
$\Pfpf^+(y,p) \subset \Pfpf^+(y,q) $ and $ \Pfpf^-(y,q) \subset \Pfpf^-(y,p) .$
\end{lemma}

\begin{proof}
If $y \lessdot_\F z = (p,i)y(p,i)$ where $p<i$ then $q=z(p)<z(i)$ by Proposition~\ref{fpfbruhatcover-thm}
and $z=(q,z(i))y(q,z(i))$; hence $\Pfpf^+(y,p) \subset \Pfpf^+(y,q) $. The other inclusion follows similarly.
\end{proof}

For a simple transposition $s \in S_\ZZ$  and $X\subset \F_\ZZ$,
let
$\cD_s(X) = \{ szs : z \in X \text{ and }s \in \DesF(z)\}$.
Observe that if $s=s_i$ for $i \in \PP$ and $X \subset \F_\infty$
then   $\partial_i \sum_{z \in X}\Sfpf_z  = \sum_{z \in \cD_s(X)} \Sfpf_z$
by \eqref{fpf-eq}.

\begin{lemma} \label{fpfmonk-lem2} 
Suppose $y \in \F_\ZZ$,  $(p,q) \in \Cyc_\ZZ(y)$, and $s \in\DesF(y)$.
Define
$\cE^-$ (respectively, $\cE^+$) to be $\{y\}$ if $s \in \{s_p,s_q\}$
 (respectively, if $s \in \{s_{p-1},s_{q-1}\}$) and $\varnothing $ otherwise.
Then:
\ben
\item[(a)]
$  \Pfpf^-(sys,s(p)) =  \cD_s(\Pfpf^-(y,p)) \cup \cE^-$

\item[(b)]
$ \Pfpf^+(sys,s(q)) =  \cD_s(\Pfpf^+(y,q)) \cup \cE^+.$

\een
\end{lemma}

Note that the unions expressing $\Pfpf^-(sys,s(p))$ and $ \Pfpf^+(sys,s(q)) $
in this lemma are  disjoint.

\begin{proof}
If $z \in \Pfpf^-(y,p)$ is such that $s \in \DesF(z)$ and  $i<p$ is such that
$ z=(i,p)y(i,p)$, then  $szs \neq y$ and  $s\neq (i,p)$
since  $s \in \DesF(y)$, and  from this observation it is routine to show that
$szs\in \Pfpf^-(sys,s(p))$.
Alternatively, 
suppose $z \in \F_\ZZ$ is such that $szs \in \Pfpf^-(sys,s(p))$,
so that   $z = (s(i),p)y(s(i),p)$ for some $i<s(p)$, and assume $szs\neq y$.
It cannot hold that $s = (s(i),p)$ since $szs\neq y$, so we have $s(i)<p$.
Moreover, since $s(szs)s = z \neq sys$, 
Lemma \ref{fpfmonk-lem1}
implies that  $s\notin\DesF(szs)$,  so $s \in \DesF(z)$ and $\ellfpf(z) = \ellfpf(y)+1$.
We deduce in this case that $z \in \Pfpf^-(y,p)$ and $szs \in \cD_s(\Pfpf^-(y,p))$.

The previous paragraph shows that $\Pfpf^-(sys,s(p)) - \{y\} =\cD_s(\Pfpf^-(y,p))$.
Since $s$ cannot be the transposition $(p,q)$, it is evident that $y \in \Pfpf^-(sys,s(p))$
if  $s \in \{s_p,s_q\}$.
Suppose  $s=s_n$ for some $n \in \ZZ$ and  $y \in \Pfpf^-(sys,s(p))$.
It remains to show that $n \in \{p,q\}$.
Let $i < s(p)$ be such that  $y = (i,s(p))sys(i,s(p))$.
Since
 $y(n) > y(n+1)$,
it follows from Lemma \ref{fpfbasic-lem} that  either $i=n<s(p)=n+1$ or $i=y(n+1)< s(p) = y(n)$.
In the first case we have $p =s(n+1) = n$ as desired; in the second case, we must have
$p \notin \{n, n+1\}$ (since if $p=n$ then $s(p)=y(n)$ would imply that $q = p+1$ and
$(n,n+1) \in \Cyc_\ZZ(y)$, contradicting the assumption that $s=s_n \in \DesF(y)$, while if
$p=n+1$ then $y(n+1)<s(p)$ would imply that $q<p-1$, contradicting the assumption that $p<q$),
so $p=s(p) = y(n)$ and $n= y(p) = q$, as needed.
This proves the desired formula for $\Pfpf^-(sys,s(p))$,
and the analogous identity for $ \Pfpf^+(sys,s(q))$ follows by similar arguments.
\end{proof}

The promised transition formula for $\Sfpf_z$ now goes as follows.

\begin{theorem}\label{thm:fpf-transition-formula}  If $y \in \F_\infty$ and $(p,q) \in \Cyc_\PP(y)$ then
\[
(x_p+x_q) \Sfpf_y = \sum_{ z \in \Pfpf^+(y,q)} \Sfpf_{z} -  \sum_{z \in \Pfpf^-(y,p)} \Sfpf_{z}
\]
where we set $\Sfpf_z = 0$ for $z \in \F_\ZZ- \F_\infty$.
\end{theorem}

\begin{example}
As per Example~\ref{ex:fpf-phi}, we have
\begin{equation*}
  (x_3+x_6)\Sfpf_{(1,5)(2,4)(3,6)(7,8)} = \Sfpf_{(1,5)(2,4)(3,7)(6,8)} - \Sfpf_{(1,5)(2,6)(3,4)(7,8)} - \Sfpf_{(1,6)(2,4)(3,5)(7,8)}.
\end{equation*} 
%$(x_3+x_7)\Sfpf_{(1,2)(3,7)(5,4)(6,8)} = \Sfpf_{(1,2)(3,8)(4,5)(6,7)}-\Sfpf_{(1,3)(2,7)(5,4)(6,8)}$.
\end{example}

\begin{remark}
Our proof of  Theorem~\ref{thm:fpf-transition-formula} is somewhat different from its counterpart in Section~\ref{itransition-sect}.
Rather than expanding both sides using the original transition formula (Theorem~\ref{thm:ordinary-transition-formula}),
we show by induction that both sides have the same image under each divided difference operator and are therefore equal.
The first approach would work just as well here as in Section~\ref{itransition-sect},
but it seems useful to present an alternate method.
On the other hand, it appears difficult to use
a divided differences argument 
to prove Theorem~\ref{invmonk-thm}, as the required case analysis is very complicated.
\end{remark}

\begin{proof}[Proof of Theorem~\ref{thm:fpf-transition-formula}]
We prove the result
 by induction on $\ellfpf(y)$. 
When $\ellfpf(y)=0$, $y=\wfpf$ so  $p=2i-1$ and $q=2i$ for some $i \in \PP$, and
the theorem  reduces to the true equation $ x_{2i-1} + x_{2i} = \fkS_{s_{2i}} - \fkS_{s_{2i-2}}$.
 Assume $\ellfpf(y) > 0$, let $i \in \PP$ and $s=s_i$,
 and define
 $\epsilon(i,p,q)$ to be 1 if $i \in \{p,q\}$, $-1$ if $i \in \{p-1,q-1\}$,
 and 0 otherwise.
 Since $i \in \PP$ is arbitrary and
since only constant polynomials are annihilated by every divided difference operator,
it suffices to show that the homogeneous polynomials 
$F=(x_p+x_q) \Sfpf_y $
and 
$G=\sum_{ z \in \Pfpf^+(y,q)} \Sfpf_{z} -  \sum_{z \in \Pfpf^-(y,p)} \Sfpf_{z}$
have the same image under $\partial_i$.

If $s \notin \DesF(y)$ so that $\partial_i \Sfpf_y = 0$,
 then    Lemma \ref{lem6} implies
$
 \partial_i F
 =
 \epsilon(i,p,q) \Sfpf_{y} 
$,
which  coincides 
  with $\partial_i G$ 
by \eqref{fpf-eq} and  Lemmas \ref{fpfbasic-lem} and \ref{fpfmonk-lem1}.
Assume $s \in \DesF(y)$ so that $\partial_i \Sfpf_y = \Sfpf_{sys}$ and $(i,i+1)\neq (p,q)$.
We then have $(s(p),s(q)) \in \Cyc_\PP(sys)$,
so we deduce by Lemma \ref{lem6} and induction that 
$
\partial_i F
=
\epsilon(i,p,q) \Sfpf_{y}
+ \sum_{z \in \Pfpf^+(sys,s(q))} \Sfpf_z
- \sum_{z \in \Pfpf^-(sys,s(p))} \Sfpf_z.
$
It follows by  Lemma \ref{fpfmonk-lem2} that
the last expression is again equal to
$\partial_i G$.
\end{proof}

One  application of the preceding theorem is a self-contained algebraic proof
of the following product formula for $\Sfpf_{w_n} $ when $n$ is even.
This formula is due originally to Wyser and Yong, who take it as the definition of
$\Sfpf_{w_n} = \Upsilon_{w_n; (\GL_{n}, \Sp_n)}$ in \cite{WY}.
Our remarks after Corollary \ref{wy-cor} about Wyser and Yong's construction
apply equally well in this context.

\begin{corollary}[Wyser and Yong \cite{WY}]
\label{wy-fpf-cor}
If $n \in 2\PP$ then
$\Sfpf_{w_n} =\ds \prod_{\substack{1\leq i < j \leq n  \\ i+j \leq n}} (x_i+x_j)$.
\end{corollary}

Like Corollary \ref{wy-cor}, 
this formula is a special case of more general result \cite[Eq.\ (18)]{CJW2}.

\begin{proof}
The proof is similar to that of Corollary \ref{wy-cor}.
We may assume that $n>2$.
Define $u_0=w_{n-2}s_{n-1}$ and $u_{2i-1} = s_{n-1-i} u_{2i-2} s_{n-i}$
and $u_{2i} = s_{n-i} u_{2i-1}  s_{n-i}$ for $i \in [ \frac{n}{2}-1]$. Then $u_{n-2}=w_n$,
and one derives from Theorem \ref{thm:fpf-transition-formula}
that $(x_i+x_{n-1-i}) \Sfpf_{u_{2i-2}} = \Sfpf_{u_{2i-1}}$ and
$(x_i+x_{n-i})\Sfpf_{u_{2i-1}} = \Sfpf_{u_{2i}}$ for $1 \leq i \leq \frac{n}{2}-1$.
Since $  \Sfpf_{u_0} = \Sfpf_{w_{n-2}}$, the desired formula follows by induction.
\end{proof}

\subsection{Another extension of the Little map}\label{little2-sect}

In this section we construct a fixed-point-free variant of
the involution Little map in Section~\ref{little-sect},
in order to give a bijective proof of the following identity.
As with Proposition~\ref{little3-cor}, this can also be derived algebraically from the corresponding transition formula.

\begin{proposition}\label{stanleyfpf-cor}
If $y \in \F_\ZZ$ and $(p,q) \in \Cyc_\ZZ(y)$,
then
\[\sum_{z \in \Pfpf^-(y,q)} |\cRfpf(z)|  = \sum_{ z \in \Pfpf^-(y,p)} |\cRfpf(z)|.\]
\end{proposition}

%\begin{proof}
%Fix $y \in \F_\ZZ$ and $(p,q) \in \Cyc_\ZZ(y)$.
%The map $z \mapsto z\gg N$ is a bijection $\Pfpf^+(y,q) \to  \Pfpf^+(y\gg N, q+N)$
%and
%$\Pfpf^-(y,p) \to \Pfpf^-(y\gg N, p+N)$ for all $N \in 2\ZZ$,
%so it follows by Theorem \ref{thm:fpf-transition-formula}  that 
%\be\label{little-eq}
% \sum_{ z \in \Pfpf^+(y,q)} \Sfpf_{z\gg N}  =
% \(x_{p+N} + x_{q+N}\) \Sfpf_{y\gg N} +\sum_{z \in \Pfpf(y,p)} \Sfpf_{z\gg N}
% \ee
%for all even integers $N$ with $y \gg N \in \F_\infty$.
%It is a consequence of \cite[Theorem 1.1]{BJS} that 
%if $z \in \F_\ZZ$ and $k= \ellfpf(z)$, then
%  $ |\cRfpf(z)|$ is the coefficient of $x_1x_2\cdots x_k$
%in $\Sfpf_{z \gg N}$ for all sufficiently large even integers $N$;
%see the discussion after \cite[Eq.\ (1.3)]{HMP1}.
%Our identity therefore follows 
%by comparing the coefficients of $x_1x_2\cdots x_k$
%on both sides of \eqref{little-eq} when $k = \ellfpf(y)+1$
%and $N$ is large enough.
%\end{proof}

The results and definitions here are mostly  analogous to those in Section~\ref{little-sect},
though we will encounter a few complications that are unique to the fixed-point-free setting.
To start, 
let $\a = (s_{a_1},s_{a_2},\dots,s_{a_k})$ be a sequence of simple transpositions. Fix $i \in [k]$
and define $\del_i(\a)$ as in Section \ref{little-sect}.
The pair $(\a, i)$ is a \emph{$y$-marked FPF-involution word} for some  $y \in \F_\ZZ$
if $\del_i(\a)\in \cRfpf(y)$. If $\a \in \cRfpf(z)$ for some $z \in \I_\ZZ$,
then we say that $(\a,i)$ is \emph{reduced}.
If $u^{-1}\wfpf u = v^{-1}\wfpf v $ 
for $u = s_{a_1}s_{a_2}\cdots s_{a_k}$
and 
$v=s_{a_1}s_{a_2} \cdots \wh{s_{a_i}} \cdots s_{a_k}$,
then we say that $(\a,i)$ is \emph{semi-reduced}.
As in Section~\ref{little-sect},
$(\a,i)$ is \emph{nearly reduced} if $\a \in \cR(w)$ for some
$w \in S_\ZZ$ but $\a \notin \cRfpf(z)$ for all $z \in \F_\ZZ$. 
The first thing to note about this terminology is the following variant of Lemma~\ref{invlem:nearly-reduced}.

\begin{lemma} \label{fpflem:nearly-reduced}
Let $y \in \F_\ZZ$ and suppose $(\a,i)$ is $y$-marked FPF-involution word which is neither reduced nor semi-reduced.
There is then a unique index $j \neq i$ such that $(\a,j)$ is a $y$-marked FPF-involution word.
The new $y$-marked word $(\a,j)$ is also not semi-reduced.
\end{lemma}

\begin{proof}
Write $\a = (s_{a_1},s_{a_2},\dots,s_{a_k})$.
As noted earlier, $(\F_\ZZ, \ellfpf)$ is a quasiparabolic $S_\ZZ$-set in the sense of Rains and Vazirani
by 
\cite[Theorem 4.6]{RainsVazirani}.
Given this fact, the first assertion
is equivalent to  \cite[Corollary 2.12]{RainsVazirani} with $x_0 = \wfpf$, 
$w=s_{a_1}\cdots \wh{s_{a_i}} \cdots s_{a_k}$,
and $r = s_{a_k}\cdots s_{a_{i+1}}s_{a_i} s_{a_{i+1}}\cdots s_{a_k}$.
The second assertion holds since if $(\a,j)$ were semi-reduced then we would have 
$y = w^{-1}\wfpf w$ for $w = s_{a_1}\cdots \wh{s_{a_{i}}} \cdots \wh{s_{a_{j}}} \cdots s_{a_k}$,
contradicting the  implicit assumption that $\ellfpf(y) =k-1$.
\end{proof}

%\begin{remark}
%In contrast to the situation for Lemma~\ref{invlem:nearly-reduced}, here it is straightforward to locate the relevant index $j\neq i$
%from the corresponding wiring diagram.
%Consider a $y$-marked FPF-involution word $(\a,i)$, and let $\cD(\a)$ be the FPF-involution wiring diagram of
%$\a = (s_{a_1},s_{a_2},\dots,s_{a_k})$ as defined at the beginning of this section.
%In $\cD(\a)$, let the \emph{wire} of $t \in \ZZ$ refer to the path starting at $(k,t)$ and ending in the 0th row $\cD(\a)$,
%and let the \emph{arc} of $t$ refer to the path formed by the wires of both $t$ and $y(t)$.
%If $\cK$ is the wire of $t$, then we let $y(\cK)$ denote the wire of $y(t)$.
%
%Now suppose $\cK$ and $\cL$ are the wires crossing in the $i$th row of $\cD(\a)$.
%If $(\a,i)$ is neither reduced nor semi-reduced,
%then $\cL \neq y(\cK)$ and the wires 
%$\cK$ and $\cL$ or the wires $y(\cK)$ and $y(\cL)$ cross in another row; the index of this row is 
%the unique index $j\neq i$ described in the lemma.
%\end{remark}

Semi-reduced and nearly reduced marked words have the following relationship:

\begin{lemma}\label{semi-reduced-lem}
All semi-reduced $y$-marked FPF-involution words are nearly reduced.
\end{lemma}

\begin{proof}
Let $y \in \F_\ZZ$ and suppose $(\a,i)$ is a semi-reduced $y$-marked FPF-involution word.
As usual write $\a = (s_{a_1}, s_{a_2},\dots, s_{a_k})$,
and let $t$ be the transposition $s_{a_1}s_{a_2}\cdots s_{a_i}\cdots s_{a_2} s_{a_1}$.
The definition of semi-reduced implies that $t$ commutes with $\wfpf$,
so $t = s_j$ for some odd integer $j$ and 
$ s_{a_1} s_{a_2}\cdots s_{a_k} = s_j u$ for $u = s_{a_1}\cdots \wh{s_{a_i}}\cdots s_{a_k}$. Since $u \in \cAfpf(y)$
and since $(s_ju)^{-1} \wfpf (s_j u) = u^{-1} \wfpf u = y$,
it must hold that $\ell(s_ju) =\ell(u) + 1$, so $\a \in \cR(s_ju)$ and therefore
$(\a,i)$ is nearly reduced.
\end{proof}

Next, we have this analogue of Lemma~\ref{nr-lem2}:

\begin{lemma}\label{fpfnr-lem2}
Suppose $y \in \F_\ZZ$ and $(\a,i)$ is a nearly reduced $y$-marked FPF-involution word.
Let $j=i$ if $(\a,i)$ is semi-reduced, and otherwise
let $j\neq i$ be the unique index such that $(\a,j)$ is  a $y$-marked FPF-involution word.
Define
 $u,v \in \cAfpf(y)$ and $w \in S_\ZZ$ as the permutations such that
 \[
 \del_i(\a) \in \cR(u) \qquand \del_j(\a) \in \cR(v)
 \qquand \a \in \cR(w).
 \]
If  $(p,q) \in \Cyc_\ZZ(y)$ is such that $w \in \Phi^\pm(u,p) \cup \Phi^\pm(u,q)$,
then $w \in \Phi^\mp(v,p) \cup \Phi^\mp(v,q)$.
\end{lemma}

\begin{proof}
First assume $w \in \Phi^+(u,p) \cup \Phi^+(u,q)$
so that $u\lessdot w = u(r,s)$ for some integers $r<s$ with $r\in \{p,q\}$. 
Note that by construction, $w = v(r',s')$ for integers $r'<s'$.
If $j\neq i$ then necessarily \[w^{-1}\wfpf w = (r,s)y(r,s) =(r',s')y(r',s')< y\] so
it follows by Lemma~\ref{fpfbasic-lem} 
that $\{ r',s' \} = \{ y(r),y(s)\}$.
In this case, 
Corollary~\ref{fpf<-cor} implies that $y(s)<y(r)$
so $w \in \Phi^-(v,p) \cup \Phi^-(v,q)$ as desired. 
Suppose alternatively that $(\a,i)$ is semi-reduced, so that $j=i$ and $u=v$ and $(r,s)=(r',s')$.
We then have  $w^{-1}\wfpf w = (r,s)y(r,s) = y$,
so $(r,s)$ is a cycle of $y$.
Since $r \in \{p,q\}$, it follows that $r=p$ and $s=q$,
so $w \in  \Phi^-(u,q)= \Phi^-(v,q).$
A symmetric argument shows that $w \in \Phi^+(v,p) \cup \Phi^+(v,q)$ when $w \in \Phi^-(u,p) \cup \Phi^-(u,q)$.
\end{proof}

We may now introduce fixed-point-free versions of the operators $\push$ and $\pull$ from Section~\ref{little-sect}.
Fix a $y$-marked FPF-involution word $(\a,i)$, and let $\b$ denote the sequence formed by
decrementing the index of the $i$th transposition in $\a$.
We define
$(\a,i) \ipush = (\b,j)$
where  $j$ is given  as follows:
\begin{itemize}
\item If $(\b,i)$ is reduced or semi-reduced, then $j=i$.
\item Otherwise, $j$ is the  index distinct from $i$ such that
$(\b,j)$ is a $y$-marked FPF-involution word.
\end{itemize}
The index $j$ is well-defined and uniquely determined by Lemma \ref{fpflem:nearly-reduced}.
The operation $\ipush$ is invertible and we denote its inverse by $\ipull$.
%, so that $(\b,j)\ipull = (\a,i)$ when $(\a,i)\ipush = (\b,j)$.
Explicitly, if $\b = (s_{b_1},\dots,s_{b_k})$ and $j\in[k]$
are such that $(\b,j)$ is a $y$-marked FPF-involution word,
then we have $(\b,j)\ipull = (\a,i)$ where
 $i=j$ if $\b$ is reduced or semi-reduced
and where otherwise $i$ is the unique index distinct from $j$ such that $(\b,i)$ is a $y$-marked FPF-involution word,
and in both cases $\a = (s_{b_1},\dots,s_{b_{i-1}}, s_{b_i+1}, s_{b_{i+1}},\dots,s_{b_k})$.

\begin{example}\label{fpf-ex}
With our notation as in Example \ref{little-ex}, the sequence $45\boxed{4} = ((s_4,s_5,s_4),3)$ is a semi-reduced $y$-marked FPF-involution word
for $y = \ifpf(216543) \in \F_\infty$, and we have 
\[ 45\boxed{4} \ipush = 4\boxed{5}3, \qquad 4\boxed{5}3\ipush = \boxed{4}43, \qquad \boxed{4}43\ipush =\boxed{3}43,
\qquand \boxed{3}43\ipush= \boxed{2}43.
\]
In the other direction, $45\boxed{4}\ipull = 45\boxed{5}$ and $45\boxed{5}\ipull = 4\boxed{6}5$.
Note that $\boxed{2}43$ and $ 4\boxed{6}5$ are reduced as $y$-marked FPF-involution words.
\end{example}

The index $i$ may be unchanged on applying $\ipush$ to $(\a,i)$, but 
this cannot happen more than once in succession due to 
the following lemma.

\begin{lemma}\label{pre-fpf-little-lem}
If $(\a,i)$ is a $y$-marked FPF-involution word,
then at most one of  $(\a,i)$ or $(\a,i)\ipush$ is semi-reduced.
\end{lemma}

\begin{proof}
If $(\a,i)$ and $(\a,i)\push$ were both semi-reduced, then it would hold for $j=a_i-1 \in \ZZ$
and $z = s_{a_{i-1}}\cdots s_{a_2} s_{a_1} \wfpf s_{a_1} s_{a_2} \cdots s_{a_{i-1}} \in \F_\ZZ$ that
$z = s_jzs_j = s_{j+1}zs_{j+1}$, which is impossible since $(j,j+1)$ and $(j+1,j+2)$ cannot
simultaneously be
cycles of any fixed-point-free involution.
\end{proof}

As usual, we write $\ipush^N$ and $\ipull^N$ for the $N$-fold iteration of $\ipush$ and $\ipull$.

\begin{lemma}\label{fpf-little-prop}
If $(\a,i)$ is a $y$-marked FPF-involution word,
then there are positive integers $M,N>0$ such that the $y$-marked
FPF-involution words $(\a,i)\ipush^M $ and $(\a,i)\ipull^N$
are  reduced.
\end{lemma}

\begin{proof}
The main idea is unchanged from the proof of Lemma~\ref{little-prop}.
Let  $(\a,i)$ be a $y$-marked FPF-involution word
and define $(\b,j) = (\a,i)\ipush$ and $(\c,k)=(\b,j)\ipush$.
By
Lemma~\ref{pre-fpf-little-lem}, if $(\b,j)$ or $(\c,k)$ are both not reduced,
then either $s_{a_i-1}$ appears in $\del_j(\b) \in \cRfpf(y)$
or $s_{a_i-2}$ appears in $\del_k(\c) \in \cRfpf(y)$.
Hence, if $(\a,i)\ipush^M $ is not reduced for all $M>0$, then
$\cRfpf(y)$ contains elements composed of simple tranpositions with
arbitrarily small indices, so $|\cRfpf(y)| = \infty$.
One can derive in several ways, however, that
the set $\cRfpf(y)$ is manifestly finite; this follows from \cite[Theorem 7.2]{HMP2},
for example.
A similar argument  shows that $(\a,i)\ipull^N$ is reduced for some $N>0$.
\end{proof}

\begin{definition}
The \emph{FPF-involution Little bump} $\fB$ is the operation on marked FPF-involution words defined by
$\fB(\a,i) = (\a,i)\ipush^N$
where $N> 0$ is minimal such that $(\a,i)\ipush^N$ is reduced.
\end{definition}

As with $\iB$ in Section~\ref{little-sect},
 since $\ipush$ is an invertible operation, $\fB$ is also
invertible.

\begin{example}
By Example~\ref{fpf-ex}, we have 
$\fB(4\boxed{6}5) = 4\boxed{6}5 \ipush^6 =\fB(45\boxed{4}) = 45\boxed{4} \ipush^4 = \boxed{2}43.$
\end{example}

By Theorem \ref{fpfbruhat-thm}, if $y,z \in \F_\ZZ$ are such that $y\lessdot_\F z$ then for any 
$\a\in \cRfpf(z)$, there exists a unique index $i$ such that $(\a,i)$ is a
reduced
$y$-marked FPF-involution word;  in this situation we define $\fB_y(\a) = \b$ where $\fB(\a,i) = (\b,j)$.
%One we can identify the index $i$ using wiring diagrams as follows: let $t = (k,l)$ be the unique transposition such that $z = t y t$ and $k < y(k)$.
%Then $i$ is the index of the row where the arcs of $k$ and $l$ cross in the FPF-involution diagram of $\a$.
As with the involution Little map defined in Section~\ref{little-sect}, it is clear that $\fB_y$ induces a bijection
$\bigcup_{ z } \cRfpf(z) \to \bigcup_{z} \cRfpf(z)$, with both unions  over $z \in \F_\ZZ$
with $y\lessdot_\F z$.
Our main result concerning
the \emph{FPF-involution Little map} given by $\fB_y$
is the following variant of Theorem~\ref{little-thm}, which directly implies Proposition~\ref{stanleyfpf-cor}.

\begin{theorem}
Let $y \in \F_\ZZ$ and $(p,q) \in \Cyc_\ZZ(y)$.
The map  $\fB_y$   restricts to a bijection
\[ \bigcup_{z \in \Pfpf^+(y,q)} \cRfpf(z) \to \bigcup_{z \in \Pfpf^-(y,p)} \cRfpf(z).\]
\end{theorem}

\begin{proof}
The theorem follows by nearly the same 
argument as the one given to show Theorem~\ref{little-thm},
\emph{mutatis mutandis}. In detail, define $X^\pm(u;p,q)$ for $u \in S_\ZZ$ as in the proof of Theorem~\ref{little-thm},
and let $\cB$ be the operation on marked FPF-involution words given by  $\cB(\a,i) = (\a,i)\ipush^M$ where $M>0$ is the least positive integer such that
$ (\a,i)\ipush^M$ is reduced or nearly reduced.
Although, \emph{a priori}, this definition of $\cB$ appears to be different from the one in Theorem~\ref{little-thm},
it is a consequence of Lemma~\ref{semi-reduced-lem} that $\cB$ again coincides with the ordinary {Little bump} described in \cite{LamShim}.
Let $\a \in \bigcup_{z \in \Pfpf^+(y,q)} \cRfpf(z)$
and write $i$ for the index such that $\del_i(\a) \in \cRfpf(y)$.
As in our earlier proof, define $(\a_0,i_0) = (\a,i)$ and $(\a_t,i_t) = \cB(\a_{t-1}, i_{t-1})$
for $0<t\leq k$, where $k$ is the first index such that $(\a_k, i_k)$ is a
reduced $y$-marked FPF-involution word, and let $u_t \in \cAfpf(y)$ be such that $\del_{i_t}(\a_t) \in \cR(u_t)$.

By hypothesis,
$\a \in \cRfpf(z)$ for some $z \in \F_\ZZ$ satisfying $y\lessdot_\F z = (q,j)y(q,j) = (p,i)y(p,i) $
for integers $i,j$  with $q<j$ and $p<i=y(j)$.
In view of Lemma~\ref{fpfbasic-lem},
we deduce that $\a_0 =\a \in X^+(u_0;p,q)$.
From this fact, it follows by the same inductive argument as in the proof of Theorem~\ref{little-thm}, 
but appealing to Lemma \ref{fpfnr-lem2} instead 
of  Lemma \ref{nr-lem2}, 
that
$ \a_{t} \in X^-(u_{t-1};p,q) \cap X^+(u_{t};p,q)
$
 for  $ 0< t <k$
and  $\a_k \in X^-(u_k;p,q)$.
Since $(\a_k,i_k)$ is a reduced $y$-marked FPF-involution word, Lemma~\ref{psi+lem}
implies that $\fB_y(\a) = \a_k \in \cRfpf(z)$ for some $z \in \Pfpf^-(y,p)$. Thus 
$\fB_y$  restricts to an injective map
$\bigcup_{z \in \Pfpf^+(y,q)} \cRfpf(z) \to \bigcup_{z \in \Pfpf^-(y,p)} \cRfpf(z)$.
By a symmetric argument applied to the inverse of $\fB_y$, we deduce that this map is a bijection. 
\end{proof}

\end{document}